\documentclass[a4paper,12pt]{article}

\usepackage{natbib}
 \usepackage[latin1]{inputenc}  
 \usepackage[T1]{fontenc}  
 
\usepackage{amsmath,amsfonts,amssymb,amsthm,mathrsfs}

\usepackage{cancel,ulem}
\allowdisplaybreaks
\usepackage{graphicx,xcolor}
\usepackage{stmaryrd}
\usepackage{scalerel}
\usepackage{graphicx}
\newcommand{\nback}[1][-.95pt]{
  \mathrel{\raisebox{#1}{$\rotatebox[origin=c]{-315}{\scaleobj{0.55}{-}}$}}
}
\newcommand{\precneq}{%
\mathrel{\ooalign{$\preccurlyeq$\cr\kern1.2pt$\nback$}}}

\usepackage[scr=boondox,  
            cal=esstix]   
           {mathalpha}
 \usepackage{fancyhdr}
 
  \newcommand{\m}{    \mathsf m}

 \newcommand{\X}{   \mathsf X}

\RequirePackage[colorlinks=true, linkcolor = red,citecolor=blue,urlcolor=blue,backref=page]{hyperref}

\newtheorem{corollary}{Corollary}
\newtheorem{definition}{Definition}

\theoremstyle{definition} 
\newtheorem{example}{Example} 

\newtheorem{lemma}{Lemma}
\newtheorem{proposition}{Proposition}
\newtheorem{remark}{Remark}
\newtheorem{theorem}{Theorem}
 \newcommand{\E}{  \mathscr  E}  
 \renewcommand{\sc}{  \mathscr  S}  
 \newcommand{\conv}{\textrm{conv}}  

   \renewcommand{\k}{   \mathsf k}  
   \newcommand{\F}{  \mathscr  F }  
\renewcommand{\P}{\mathsf{P}}
\newcommand{\M}{\mathsf{M}}

\newcommand{\Leb}{\mathcal{L}}

\newcommand{\C}{\mathcal C}

\newcommand{\var}{\text{Var}}

  \newcommand{\s}{  \mathscr  s}

\newcommand{\remove}[1]{}
 \newcommand{\sinc}{\textrm{sinc}}  
\renewcommand{\sp}{\textrm{sp}}

\title{Rigidity of random stationary measures and applications to point processes}

\author{Rapha\"el Lachi\`eze-Rey\footnote{\texttt{raphael.lachieze-rey\@math.cnrs.fr}, Inria Paris, France, and Lab. MAP5, Universit\'e Paris Cit\'e, France}}

\date{}

\renewcommand{\c}{   \mathsf c}
\renewcommand{\C}{   \mathsf C}
\renewcommand{\S}{   \mathsf S}

\begin{document}
\maketitle

%

 {\bf Abstract } The  {\it number rigidity} of a stationary point process $ \P$ entails that for a bounded set $ A$ the knowledge of $\P $ on $ A^{c}$ a.s. determines $\P(A)$; the $ k$-order rigidity  means  the  moments of $ \P1_{A}$  up to order $ k$ can be recovered.
We show that
  $ k$-rigidity occurs if the continuous component $ \s$ of $ \P$'s  {\it structure factor} has a zero of order $ k$ in $ 0$, by exploiting a connection with  Schwartz's Paley-Wiener theorem for analytic functions of exponential type; these results apply to any random $ L^{2}$ wide sense stationary measure   on $ \mathbb{R}^{d}$ or $ \mathbb{Z} ^{d}$.
In the continuous setting, these local conditions are also necessary   if   $ \s_{}$ has finitely many zeros, or is isotropic, or  at the opposite separable. This  explains why no model seems to exhibit rigidity in dimension $ d\geqslant 3$, and allows to efficiently recover many recent rigidity results about point processes. For a field on $ \mathbb{Z} ^{d}$, these results hold provided  $  \# A  >2k$. 
For a continuous Determinantal point process with reduced kernel $ \kappa $,   $ k$-rigidity is equivalent to  $ (1- \widehat {\kappa ^{2}})^{-1}$ having a zero of order $ k$ in $ 0$, which answers questions on completeness and number rigidity. We also deduce some non-integrability results in the less tractable realm of Riesz gases. Finally, we are able to prove that random stationary quasicrystals are maximally rigid on any compact.\\

  
 \label{ch:sufficiency}
 
  
 \section*{Introduction}

The first instance of a linear prediction problem for a stationary process goes back to   \cite{Sz20},  for a random time series $ \{\X_{k};k\in \mathbb{Z} \}$. He gave a necessary and sufficient condition for  the process to be {\it deterministic:}   {\it the future is entirely determined by the past}, i.e.
\begin{align*}
\sigma (\X_{k};k>0)\subset \sigma (\X_{k};k\leqslant 0)
\end{align*}through linear interpolation, if and only if the spectral density $ \s_{}$ satisfies 
\begin{align*}
\int_{  \mathbb  T  }\ln( \s_{}(u)  ) = -\infty 
\end{align*}
where $  \mathbb  T   = [-\pi ,\pi ].$
\cite{Kolmo-rigid}  further studies the non-deterministic case, and in particular addresses the weaker linear interpolation problem of whether $ \X_{0}\in \sigma (\X_{k};k\neq 0)$. He proves that it is so  if and only if 
\begin{align}
\label{eq:kolmo-condition}
\int_{   \mathbb  T   }\s_{}^{-1}(u)du = \infty .
\end{align}
A detailed account of the line of research about random linear interpolation can be found in  \cite{Rozanov}.
Much more recently, other similar problems emerged for higher-dimensional continuous models with the notion of  {\it rigidity} for a  {\it stationary point process}, i.e. a random   locally finite set of points $ \P\subset \mathbb{R}^{d}$ which law is invariant under spatial translations. The oldest such result might be   \cite{AizMart}, about 1D Coulomb gases, but the systematic study of rigidity really started with the introduction of the notion of  {\it tolerance}   (\cite{LyonsSteif,HolSoo}), then \cite{GP17} coined the term  {\it number rigidity} as the property that 
\begin{align*}\#\P\cap B(0,1)\in \sigma (\P\cap B(0,1)^{c}).
\end{align*}
This  obviously holds for  {\it shifted lattices}, e.g. $ \P = \{\m + U;\m\in \mathbb{Z} ^{d}\}$ (where $ U$ is uniform in $ [0,1]^{d}$ and ensures that $ \P$ is invariant under translations with non-integer coordinates), as one can easily  deduce the number of lattice points ``hidden'' in   $ B(0,1)$ by only observing $ \P\cap B(0,1)^{c}$. They   realised that, somehow surprisingly, also some strongly non-lattice like models, the  {\it infinite Ginibre ensemble} $ \P_{\textrm{Gin}}$, and the zero set $ \P_{\textrm{GAF}}$ of the  {\it planar Gaussian Analytic Function}, satisfy number rigidity.

As it turns out, these two processes possess  also possess the property of  {\it hyperuniformity}, i.e. 
\begin{align}
\label{eq:hu}
\frac{ \textrm{Var}\left(\#\P\cap B(0,r)\right)}{   \Leb(B(0,r))}\xrightarrow[r\to \infty ]{}\; 0,
\end{align}
where $ \Leb$ denotes Lebesgue measure,
at the difference of standard disordered systems such as Poisson processes; this property also is strongly reminiscent of lattice-like models. 
The surprising fact that it occurs for locally disordered,  {\it amorphous} random measures, like $ \P_{\textrm{Gin}},\P_{\textrm{GAF}}$, or other Coulomb systems, has been systematically investigated by physicists since the 90's. 
This property requires some sort of long distance dependency, compatible with  locally disordered configurations, often referred to as {\it global order and local disorder}. The activity around hyperuniformity has not stopped growing until now, as this property concerns a large diversity of models involving Coulomb systems, Gaussian analytic functions, eigenvalues of random matrices, Determinantal Point processes, and has many applications. The literature is too wide to be cited exhaustively, see for instance the survey of  \cite{Tor18} for materials science, of  \cite{GosLeb} for statistical physics, or the more mathematical discussion  of  \cite{Coste}.  

 Many models have been proven to be rigid since then, mostly in the realm of hyperuniformity: some Determinantal Point Processes (DPPs)  (\cite{BufAiry,BufBalanced, BufLinear,Buf-conditional}), eigenvalues of random operators  (\cite{rigidity-schrodinger}), Pfaffian processes  (\cite{BufPfaff}), Coulomb and Riesz systems  (\cite{DHLM,ChhaibiNaj,Cha17,DereudreVasseur,AizMart}), zeros of Gaussian processes (\cite{GP17,GK21,Lac20}), stable matchings  (\cite{KLY}),  and others  (\cite{GL18,KlattLast}). In fact, the rigidity terminology is  more generally used to describe a strong form of hyperuniformity, i.e. a system of  particles  where the variance of smooth linear statistics is one or several orders of magnitude below the Poisson behaviour  (\cite{Cha17,BBNY,Ganguly-Sarkar}).
Besides the striking nature  of rigidity and its link with hyperuniformity, it has proven to be a useful property in other types of problems, such as continuous percolation  \cite{GKP}, or to establish the singularity of conditional measures  \cite{Ghosh-conditional}. \cite{OsadaRigid} recently proved a relation with diffusive dynamics of particle systems, \cite{GhoshComplete} also exploited number rigidity to show the completeness of random sets of exponential functions, an interesting property in signal processing,  \cite{LyonsSteif} study this property to show phase uniqueness for some models from statistical physics.


Stronger forms of rigidity have also emerged. For instance, for the $ \P_{\textrm{GAF}}$ model,  \cite{GP17} showed that not-only the number of particles in $ B(0,1)$ can be determined, but also the first moment, or center of mass $ \sum_{x\in \P_{\textrm{GAF}}\cap B(0,1)}x$, corresponding to {\it $ 1$-rigidity}, and that   no further moment can be determined from the observation of $ \P_{\textrm{GAF}}\cap B(0,1)^{c}$.    They proved  similarly that for the Ginibre process no moment can be determined beyond order $ 0$. Another such result appears in  \cite{DHLM} for Sine$ _{\beta }$ processes above order $ 0.$ \cite{GK21} generalise this concept to that of $ k$-rigidity, where the $ k$-th order moment is determined.

Until now, both in the physics and mathematics litterature, the precise nature of the connection between rigidity and hyperuniformity is not precisely understood, see the discussions in   \cite{GosLeb,Coste}.  It seems in particular  that rigidity is not  understood in dimensions $d\geqslant  3$,     \cite{Cha17} discusses at length this question and  proves that the  {\it hierarchical Coulomb gas}, a simplified version of the stationary 3D Coulomb gas which is not formally stationary, is not number rigid.  \cite{DereudreVasseur} conjecture that the number rigidity of a Riesz gas of index $ s$ occurs  if and only if  $ s\leqslant d-1$. \\

 {\bf Necessary and sufficient conditions.}
In the current article, we study the rigidity properties of a $ L^{2}$ (wide-sense) stationary measure $ \M$ on $ \mathbb{R}^{d}$ or $ \mathbb{Z} ^{d}$, where a point process $ \P$ is naturally associated to the atomic measure $ \sum_{x\in \P}\delta _{x}$. We shall establish for $ k\in \mathbb{N} $ a characterisation of $ k$-th order linear rigidity, which is the property that the $ k$-th order moments of    $ \M$ restricted to $ B(0,1)$ can be a.s. recovered (linearly) from $ \M  1_{ B(0,1)^{c}}$, in terms of the behaviour around the origin of the spectral measure $ \S$, the generalised Fourier transform of the covariance measure, and more precisely of its continuous component  $ \s_{}(u)du$. 
 The  deep relation between the behaviour of $ \S_{}$ around $ 0$ and hyperuniformity emerges from the fact that  \eqref{eq:hu}  is the equivalent in the Fourier domain to the vanishing of   $ \S$ near $ 0$ in a weak sense. We show exactly how the decay exponent of $ \s_{}$  is related to its degree of rigidity, i.e. the number of moments of $ \M1_{ B(0,1)}$ that can be determined by the outside configuration $ \M1_{ B(0,1)^{c}}.$
 
More precisely, on a bounded set $ A$, it is sufficient for $ k$-rigidity  that $ \s_{}^{-1}$ has a pole of order $ k$ in $ 0$ in a particular sense (Theorem \ref{thm:main-krigid-iso}). For  {\it number rigidity} ($ k = 0$), if $ \s_{}$ is  {\it isotropic}, i.e. invariant under rotations of $ \mathbb{R}^{d}$, or if $ d = 1$, it exactly means
that for any neighbourhood $ U$ of $ 0,$
\begin{align}
\label{eq:k-rigid-intro}
\int_{U }\s_{}(u)^{-1}\|u\|^{2k}du = \infty ,
\end{align}
the situation can be more complicated in higher dimensions   without isotropy (Proposition \ref{ex:spectral-counterexample}).
This explains most  results cited above about number rigidity and $ 1$-rigidity, and also generalises the sufficient conditions of  \cite{GL-sufficient} and  \cite{BufLinear} for number rigidity.

 We show that this condition is necessary  under some structural assumptions such as finite number of zeros, isotropy or separability (Proposition \ref{prop:not-lmr-finite-z} and Theorem \ref{thm:converse-rigidity}).  It also explains why rigidity does not seem to occur in dimension $ d\geqslant 3$: if $ \s_{}(u)\sim \sigma \|u\|^{2}$ for some $ \sigma >0$ as $ u\to 0$, as it is expected for many hyperuniform integrable systems, if $ d\geqslant 3,$ 
the left hand member of  \eqref{eq:k-rigid-intro} is finite for all $ k$. Hence one must find processes where the structure factor decreases faster to $ 0$ to find higher order rigidity. We exhibit in Section  \ref{sec:flower} a class of processes $ \P_{n},$ for $ n$ a prime number such that  $ \s_{}(u)\sim c_{n}\|u\|^{2n}$ for some $ c_{n}>0$. We also give general  {\it tolerance results}, meaning that if $ \s$ does not decay too fast close to $ 0$, its higher order moments are not constrained (Theorem \ref{thm:tolerant-simple}).

The converse part also allows to assess the correlation properties of a point process based on its rigidity behaviour: it is shown by  \cite{DereudreVasseur} and  \cite{DHLM} that some $ \beta $-ensembles are not $ k$-rigid for some $ k$, which implies non-integrability results on their correlation functions (see Section \ref{sec:gibbs}).

Using also Gaussian fields,  we give counter-examples of spectral measures showing that these results are optimal in full generality.

 {\bf Determinantal point processes.} The results apply efficiently to stationary DPPs. Theorem \ref{thm:k-rigid-dpp} yields that a DPP on $ \mathbb{R}^{d}$ with kernel $ K$ is   $ k$-rigid  if and only if  $ ( 1- \F({  | K | ^{2}}))^{-1}$ has a pole of order $ k$ in $ 0$, retrieving for $ k = 0$  the suffucuency conditions  of \cite{GP17,GhoshComplete,BufLinear}, and also completes the answer of  \cite{GhoshComplete} to  a question of  \cite[Section 4]{Lyons} on completeness of complex exponentials.  For DPPs on $ \mathbb{Z} ^{d}$, a pole of order $2k$ of $ \s$ in $  \mathbb  T  ^{d}$ indeed implies $ k$-rigidity, but the converse is only true in dimension $ 1$, or under some additional assumption.

 {\bf Discrete processes.} Stationary discrete processes $ \X = \{\X_{\m};\m\in \mathbb{Z} ^{d}\}$ show a similar connection between rigidity and the behaviour around $ 0$ of the spectral measure $ \S$, defined on $  \mathbb  T  ^{d}$. We show in Theorem \ref{thm:discrete} for $ A =  \llbracket m \rrbracket^{d}	: = \{-m,\dots ,m\}^{d}$ that $ \X$ is not  {\it maximally rigid} on $ A$, meaning $ \{\X(\m); |\m | \leqslant m \}\not\subset \sigma (\{\X_{\m}, | \m | >m\})$, 
  if and only if  there exists a trigonometric polynomial $ \psi (u) = \sum_{ | \m | \leqslant m}a_{\m}e^{i\langle \m, u\rangle}$ such that, 
\begin{align*}\int_{  \mathbb  T  ^{d}}\frac{  | \psi (u) | ^{2}}{\s_{}(u)}du < \infty ,
\end{align*} and not $ k$-rigid if furthermore the derivatives of $ \psi $ up to order $ k$ do not vanish in $ 0$. For this reason, we are only able in the discrete setting to give the example of a 1-rigid field that is not 0-rigid (Example \ref{ex:discrete}).

Besides the seminal work of \cite{Kolmo-rigid} in dimension $ 1$, this unifies some results established in the context of discrete DPPs, such as the result of \cite{LyonsSteif} regarding the  {\it strong full K} property for uniqueness of phase transition in statistical physics.
The condition can be made more explicit in dimension $ 1$: $ \X$ is  {\it maximally rigid} on $ \llbracket m \rrbracket$   if and only if  the number of poles of $ \s_{}$ counted with multiplicity is $ > m$, as for the condition \eqref{eq:kolmo-condition} of  \cite{Kolmo-rigid} in the case $ d = 1,m = 0,\k = 0$, or the generalisation of  \cite{BufLinear} in the case $ d = 1,\k = 0,m\in \mathbb{N}$: $ \X$ is $ 0$-rigid, i.e. $ \sum_{ | \m | \leqslant m}\X_{\m}\in \sigma (\{\X_{\k}; | \m | >m\})$,  if and only if   $ 0$ is a pole, i.e.
\begin{align*}
\int_{ B(0,\varepsilon )}\s_{}^{-1}(u)du = \infty,\varepsilon >0 ,
\end{align*}or (obviously)  if $ \X$ is  {\it maximally rigid} (Proposition \ref{prop:spec-1D}).\\

 {\bf Rigidity of the class of functions of exponential type}.
 Let us sketch here the method to derive these results for a stationary random measure $ \M$ with spectral measure $ \S$ (formally developped at Section \ref{sec:theory}). The basic idea is the following: given a function $ \gamma $ bounded on some compact $ A$, such as $ \gamma  = 1_{A}$, the linear statistic $   \int_{ }\gamma d\M$ is completely determined by $ \M1_{A^{c}}$ if
\begin{align*}
\inf_{h}
 \textrm{Var}\left(\int_{ }\gamma d\M-\int_{ }hd\M\right) = 0,
\end{align*}
for $ h$ in the space $ \mathcal{C}_{c}^{\infty }(A^{c})$  of $ \mathcal{C}^{\infty }$ functions with compact support in $ A^{c}$. By definition of $ \S$, this translates in the Fourier domain as 
\begin{align*}
\inf_{h\in \mathcal{C}_{c}^{\infty }(A^{c})}\int_{ } |  \hat \gamma - \hat h | ^{2}d\S = 0.
\end{align*}
In other words, $ \hat \gamma $ must be in the $ L^{2}(\S)$-closure of the space $ H_{A}$ spanned by the $ \hat h$ for $ h$ in $ \mathcal{C}_{c}^{\infty }(A^{c})$, or equivalently it must be orthogonal to any function $ \varphi \in H_{A}^{\perp}$. 
The crucial observation   is that for such $ \varphi ,$  the tempered distribution $ \varphi \S$ has by definition a spectrum bounded (by $ A$). This conveys a very strong form of regularity by the Schwartz-Paley-Wiener theorem (Theorem \ref{thm:SPW}):   $\psi : =  \varphi \s_{} = \varphi \S$  has no singular part and must be an analytic function of exponential type on $  \mathbb C ^{d}$. We have also immediately $ \psi \in L^{2}(\s_{}^{-1})$ because $ \int_{ }\psi ^{2}\s_{}^{-1} = \int_{ }\varphi ^{2}\S$.  If such $ \psi $ do not exist besides the null function, $ H_{A}^{\perp} = \{0\}$ and all $ \gamma $ can be predicted, meaning the process is maximally rigid. Number rigidity means that $ \gamma  \equiv 1$ is orthogonal to  $ \hat \psi $ for all such $ \psi $, i.e. for all entire function $ \psi $ with bounded spectrum, being in $ L^{2}(\s_{}^{-1})$ implies
\begin{align*}
\psi (0) =\langle 1, \hat \psi  \rangle =  0.
\end{align*}
If $ \s_{}^{-1}$ is non-integrable around $ 0$, indeed $ \psi (0) = 0$ as otherwise $ \psi \notin L^{2}(\s^{-1})$, hence $ \M$ is  number rigid.

These high regularity and integrability requirement on $ \psi $ are in some sense the source of the  {\it rigidity phenomenon}. It is then easier to characterise whether all such $ \psi \in L^{2}(\s_{}^{-1})$ must satisfy $ \langle \hat \gamma ,\psi \rangle_{\mathbb{R}^{d}} =  0$, meaning $ \gamma $-rigidity, and it mostly depends on the zeros of $ \s_{}$, or more generally on how often $ \s_{}(u)$ is close to $ 0$. What we study is not mere rigidity, but  {\it linear rigidity}, meaning that $\int_{ }\gamma d\M$ should be approximated by  linear functionals $ I_{\M}(h)$, not by  {\it any} functional of $ \M1_{A^{c}}.$ \\

 Let us present the rest of the paper. In Section  \ref{sec:spectral}, we introduce the spectral measure $ \S$ necessary to state the main results. In Section  \ref{sec:main-results}, we formally introduce the  concepts of rigidity and $ \k$-poles and derive the announced necessary and  sufficient condition. We show how they apply to some particular examples, in particular we give examples satisfying any prescribed order of rigidity, and state conversely the tolerance results; we also give the corresponding results for discrete processes. Section  \ref{sec:appli} is devoted to the applications to DPPs, Gibbs measures, and quasicrystals. Finally, Section \ref{sec:theory} gives the formal framework about tempered distributions, Schwarz' Paley-Wiener Theorem and the proofs. \\
 
 In the companion paper  \cite{rigid-companion}, we explore the maximal rigidity phenomenon, and its consequences for periodic discrete random fields, quasicrystals, and   stealthy systems. We also find   another surprising  behaviour,   {\it short range rigidity}, where  a continuous random field   with finite range dependance  exhibits a phase transition on $ A = B(0,\rho )$: maximally rigid iff $ \rho \leqslant 1$, and not rigid at all otherwise.

 \section{Spectral measure}
    \label{sec:spectral}
     \newcommand{\e}{   \mathsf E}

 Even though our main motivation is the class of point processes, we consider more generally a random $ L^{2}$  wide-sense stationary (WSS) signed measure $ \M$, i.e. a collection of real $ L^{2}$ random variables $ I_{\M}(f)$ on a probability space $ (\Omega ,\mathbf{P}),$ for $ f$ in the space $ \mathcal{C}_{c}^{b}(\mathbb{R}^{d})$ of measurable bounded and compactly supported functions, satisfying for $ f,g\in \C_{c}^{b}(\mathbb{R}^{d})$ \begin{itemize}
\item[(i)] $I_{\M}(f + g) = I_{\M}(f) + I_{\M}(g)$
\item [(ii)] $\textrm{Var}\left(I_{ \M}(\tau _{x}f) \right) =  \textrm{Var}\left(I_{\M}(f)\right)$ where $ \tau_{x}$ is the operator of translation by $ x\in \mathbb{R}^{d}$,
\item[(iii)]  letting $ \hat f(u) = \int_{}e^{-i \langle u,x \rangle} f(x)dx $, 
\begin{align}
\label{eq:phase-formula-SF}
 \textrm{Var}\left(I_{\M}(f)\right) = (2\pi )^{-d}\int_{   } | \hat f | ^{2}d\S,
\end{align}
for some non-negative symmetric measure $ \S$ on $ \mathbb{R}^{d}$ called  {\it spectral measure} ({\it symmetric} means $ \M(A) =  \dot \M(A): = \M(-A),A\in \mathcal{B}(\mathbb{R}^{d})$).
\end{itemize}
 We develop below the example of point processes and Gaussian random fields.
The class of admissible square integrable linear statistics $ I_{\M}(f)$ generally extends to a broader class of functions $ f$. 
Lemma \ref{lm:decrease-S-temp} yields that $ \S$ is a  {\it tempered distribution}, and as such it is the Fourier transform of some tempered distribution $ \C$, also symmetric, called   the  {\it correlation measure}, characterised through a simple calculation by 
\begin{align}
\label{eq:cov-measure}
 \textrm{Cov}\left(I_{\M}(f),I_{\M}(g)\right) =  \int_{ \mathbb{R}^{d}}f(y)g(x + y)\C(dx)dy,f,g\in \mathcal{C}_{c}^{b}(\e).
\end{align}
Remark that despite its name, $ \C$ is not positive, and it is not necessarily a (signed) measure, as $ \C(\mathbb{R}^{d})$ might not make sense.
Taking for instance as $ f,g$ indicator functions of small balls around  points $0$  and $ x\neq 0$, $ \C(dx)$ measures the covariance between masses around $ 0$ and $ x.$

%
%

Informally, the behaviour of $ \S$ at infinity represents the regularity of $ \M$; for instance if $ \M$ is a smooth Gaussian field, then $ \S$ will   decay fast at $ \infty $, whereas if $ \P$ is a point process, $ \S$ has likely infinite mass; the regularity of $ \S$ around $ 0$ is related to the long range dependency of $ \M$ as we will see.

\subsection{  Point processeses} Point processes are the prominent examples motivating the current work and related line of literature about rigidity. 
A  {\it point configuration} $ \eta $ is a locally finite counting measure on $ \mathbb{R}^{d}$, i.e. $\eta :\mathcal{B}(\mathbb{R}^{d})\to \mathbb{N}$, and a random point process is a random variable $ \P$ in the space of configurations endowed with the counting $ \sigma $-algebra generated by mappings $ \P\mapsto \P (A)$ for $ A$ compact.  Say that $ \P$ is  {\it simple} when it has no multiple point,  it is then sometimes assimilated to the random set formed by its support to use set-related notation such as $ \P\cap A$, as there is a one-to-one correspondance. Then the local square integrability assumption writes $ \mathbf{E}(\P(A)^{2})<\infty $ for bounded measurable $ A$. The corresponding WSS random measure is defined by $ I_{\P}(f) = \int_{}fd\P$, see  \cite{Coste} for details.
   
   The spectral measure is sometimes called  {\it structure factor} in the context of point processes.
    For instance if $ \P  $ is the unit intensity homogeneous Poisson process,  \eqref{eq:cov-measure} yields $ \C  = \Leb,$  then Plancherel formula gives $\S= \delta _{0} $ the Dirac mass in $ 0$, but this example is not of great interest for us because, due to its spatial independence,  $ \P$ does not experience rigidity or hyperuniformity.
    For point processes, long range interaction is sometimes measured through the  {\it truncated correlation measure} $ \rho ^{(2)}_{tr}(dt)$ defined by
\begin{align*}
 \rho ^{(2)}_{tr}  = \C- \delta _{0}.
\end{align*}
The Dirac mass $ \delta _{0}$ comes from  the diagonal terms in the double summation of the covariance, reflecting the purely atomic nature of the point process.

\subsection{Discrete fields and dual group $ \hat \e$ notation}
We also consider a WSS random field on $ \mathbb{Z} ^{d}$, i.e. a collection of centred variables of $ \X = \{\X(\m);\m\in \mathbb{Z} ^{d}\}$ such that 
\begin{align*}
\c(\m) : = \mathbf{E}((\X_{\m'})\X(\m + \m'))
\end{align*}
does not depend on $ \m'\in \mathbb{Z} ^{d}$, under the convention $  \textrm{Var}\left(\X(0)\right) = 1$.  Still in the sense of tempered distributions (formally applied to $ \C := \sum_{\m\in \mathbb{Z} ^{d}}\c(\m)\delta _{\m}$), the $ \c(\m)$ are the coefficients of a symmetric probability measure $ \S$ on $  \mathbb  T  ^{d}$, also called the  {\it spectral measure}, i.e.
\begin{align*}
\c(\m) = (2\pi )^{-d}\int_{  \mathbb  T  ^{d}}e^{\imath \langle \m,u \rangle}\S(du),\m\in \mathbb{Z} ^{d}.
\end{align*}
With the random measure $I_{\M}(f)$ on $  \mathbb{Z} ^{d}$ defined by $ \sum_{\m}f(\m)\X(\m)$, some arguments are very similar in $ \mathbb{R}^{d}$ and $ \mathbb{Z} ^{d}$, and we adopt the general notation $ \e\in \{\mathbb{R},  \mathbb Z  \}$.
A simple calculation shows that \eqref{eq:phase-formula-SF} is still satisfied with $  \mathbb  T  ^{d}$ as integration domain for $ f$ with finite support in $ \mathbb{Z} ^{d}$, and we denote the dual group  $ \hat \e =  \mathbb  T   $ if $ \e =   \mathbb Z , \hat \e = \mathbb{R} $ if $ \e = \mathbb{R} .$ On a purely formal side, a WSS discrete field $ \X$ can be also viewed as WSS  random measure on $ \mathbb{R}^{d}$ under the form $ \sum_{\m}X(\m)\delta _{\m + U}$, where $ U$ is independent and uniform on the dual cell $ [0,1]^{d}$.
    
\subsection{  Gaussian process}
  \label{ex:gauss} Let $ \e,\hat \e $ as above.
 Given any finite  non-negative symmetric measure $ \S$   on $ \hat \e^{d}$,  the spectral representation theorem (\cite[Th. 5.4.2]{AT07}) gives  a (mean square continuous) WSS centred Gaussian process $ \X = \{\X(x),x\in \e\}$ such that $ \S$ is the spectral measure of   $ I_{\M}(f) := \int_{ }f(x)\X(x)\mu (dx)$, where $ \mu $ is either Lebesgue ($ \e = \mathbb{R}$) or the counting measure ($ \e = \mathbb{Z} $). $ \X$ is furthermore strongly stationary in the sense where $\tau_{y}\X :=  \{\X(x + y);x\in \e\}$ has the same law as $ \X$, for $ y\in \e.$
  Since $ \S$ is finite, $ \C = (2\pi )^{-d}\F\dot\S = \mathsf c \Leb$ where the  covariance function $   \mathsf c$ is defined by 
  \begin{align*}
\mathbf{E}(\X(0)\X(x)) =   \mathsf c(x),x\in \e.
\end{align*}

  \section{Rigidity of random stationary measures}
 \label{sec:main-results}
 Let $ \e\in \{\mathbb{R},\mathbb{Z} \}.$
For $ \M$ a WSS random measure, and $ A\subset \e$, we denote by $ \M_{A}$  the collection $ I_{\M}(f)$ for $ f\in \mathcal{C}_{c}^{b}(A)$.
The general problem of rigidity   is to be able to infer a functional $ F(\M_{A} )$   knowing only $ \M_{A^{c}}$. The  {\it mass rigidity}, for instance, called  {\it number rigidity} for point processes, means that $ \M(A)\in \sigma (\M _{A^{c}})$. {\it Maximal rigidity} means that any bounded functional $ F$ $ \M_{A}$-measurable can be  predicted by $ \M_{A^{c}}$, or equivalently 
\begin{align*}\sigma (\M_{A})\subset \sigma (\M_{A^{c}}).
\end{align*}
For $ \gamma :\e\to  \mathbb C $,
 bounded with bounded support, we are interested here in predicting $ I_{\M}(\gamma )$, i.e. in determining if $ I_{\M}(\gamma )\in \sigma (\M_{A^{c}})$, called $\gamma  $-rigidity on $ A$, and maximal rigidity (MR) on $ A$  holds if $ \gamma $-rigidity holds for all such $ \gamma $. Define furthermore  {\it linear} $ \gamma $-rigidity if $ I_{\M}(\gamma )$ can be approximated by  linear statistics of $ \M_{A^{c}}$, i.e. if a.s. and in $ L^{2}(\mathbf{P})$, for some $ h_{n}\in \mathcal{C}_{c}^{b}(A^{c}),n\geqslant 1,$
\begin{align*}
I_{\M}(\gamma ) = \lim_{n}I_{\M}(h_{n}),
\end{align*}which we write $ I_{\M}(\gamma )\in \sigma _{lin}(\M_{A^{c}}),$
and  {\it linear maximal  rigidity} (LMR) if $ I_{\M}(\gamma )\in \sigma _{lin}(\M_{A^{c}})$ for all $ \gamma $  bounded. The term  {\it linear} is sometimes omitted in this article, but all methods employed here pertain to linear rigidity.
 
We are particularly interested in (linear) $ \k$-rigidity for $ \k = (\k_{i})\in \mathbb{N}^{d}$, i.e. $ \gamma _{\k}$-rigidity for $ \gamma _{\k}(t) = t^{\k} $  with $ t^{\k}= \prod_{i = 1}^{d}t_{i}^{\k_{i}},t\in \mathbb{R}^{d}$. 
For $ k\in \mathbb{N} $, say that $ \M$ is $ k$-rigid if its moments up to order $ k$ can be predicted, i.e. if  it is $ \k$-rigid for every $ \k$ with $  | \k |  \leqslant  k$ (in particular, number rigidity corresponds to $ 0$-rigidity).

Call  {\it convex body} $ A$ a compact convex set with non-empty interior. In the heart of the paper, Theorem \ref{thm:heart-article} provides an abstract necessary and sufficient condition for linear $ \k$-rigidity   on $ A$ in terms of $ \S$, and shows in particular how linear $ \gamma $-rigidity  depends on the  {\it spectral density} $ \s$, the density of the continuous part of $ \S$; $ \s$ is symmetric and non-negative, as $ \S$. For this reason, we sometimes talk about the rigidity of $ \s$ instead of the linear rigidity of $ \M$, but it means exactly the same thing. 

\subsection{Local sufficient condition}  
We consider a random stationary measure $ \M$ on $ \mathbb{R}^{d}$ with spectral density $ \s$. The crucial concept is that of a $\k$-pole.\begin{definition}

Given $ \k\in \mathbb{N}^{d}$, $ s:\mathbb{R}^{d}\to \mathbb{R}_{ + }$, say that $ \s_{}^{-1}$ has a  {\it pole of order $ \k$}  in zero, or $ \k$-pole, if it is  {locally $ u^{\k}$-incompatible} around $ 0$: for every polynomial $ Q = \sum_{\m}a_{\m}u^{\m}$  such that $\int_{ B(0,\varepsilon )} | Q | ^{2}\s_{}^{-1}<\infty $ for some $ \varepsilon >0$, we have $ a_{\k} = 0$ (with $ 1/0 = \infty $). 
\end{definition} For $ \k,\k'\in \mathbb{N}^{d}$, say that $ \k'\preceq \k$ if $ \k'_{i}\leqslant \k_{i}$ for $ i = 1,\dots ,d.$  Remark that, through multiplication of $ Q$ by $ u^{\k-\k'},$ a $\k$-pole is a $\k'$-pole for $ \k'\preceq \k.$
 Let us state our most general sufficient condition.\begin{theorem}
\label{thm:main-krigid-iso}
If $ \s_{}^{-1}$ has a $ \k$-pole  in $ 0$,  $ \M$ is $\k$-rigid on  any  bounded measurable $ A\subset \mathbb{R}^{d}$.
 \end{theorem}  
  All the proofs of results of this subsection are at  Section \ref{sec:prf-main-krigid}. 
%
In many situations,    $ 0$ being a {$\k$-pole} is equivalent to 
\begin{align}
\label{eq:ass-general-k-rigid}
\int_{ B(0,\varepsilon )}{ u^{2\k}}{\s_{}^{-1}(u)}du = \infty 
\end{align}
for all $ \varepsilon >0:$

\begin{itemize}
\item If $ \k = 0$ (number rigidity), a $ 0$-pole, or just  {\it pole,} is indeed equivalent to  \eqref{eq:ass-general-k-rigid} because if $ Q(0)\neq 0, Q(u)\sim Q(0)$ as $ u\to 0.$
\item  
In dimension $ d = 1$, we can also say that $ \s_{}$ has a  {\it zero of multiplicity $2\k$}: for every non-zero polynomial $ Q$ on $ \mathbb{R},$ we have $ Q(u) = au^{q}(1 + o(1))$ as $ u\to 0$ for some $ q\in \mathbb{N},a\neq 0$,  and $ Q\in L^{2}(\s_{}^{-1};B(0,\varepsilon ))$  if and only if  $ \int_{ B(0,\varepsilon )}u^{2q}\s_{}^{-1} <\infty $, a $ \k$-pole indeed means  \eqref{eq:ass-general-k-rigid}.

\end{itemize}

This is unfortunately not always the case:
 \begin{example}
\label{ex:order-is-complex}
 Let $ \s(u) = (u_{1}-u_{2})^{2},u\in \mathbb{R}^{2}$. Hence $ \int_{B(0,\varepsilon ) }{ u_{1}^{2}}{\s^{-1}}(u)du = \infty $ for $ \varepsilon >0$, but it does not mean that $ \s^{-1}$ has a $ (1,0)$-pole because $ Q (u) := (u_{1}-u_{2})$ satisfies $ \partial _{1}Q (0)\neq 0$ and $ \int_{ B(0,\varepsilon )}Q^{2}\s ^{-1}<\infty .$ 
 \end{example} 

We have the following more general proposition, allowing in particular to treat isotropic spectral measures.

\begin{proposition}
\label{prop:k-incomp-iso}
Define $ \tilde \s(u ) = \sup_{v:\|v\| = \|u\| }\s_{}(v ).$ Let $ \k\in \mathbb{N}^{d}.$ Then $ \M$ is $\k $-rigid on  any bounded measurable $ A$ if for $ \varepsilon >0$
\begin{align}
\label{eq:iso-k-rigid}
\int_{B(0,\varepsilon ) } \tilde \s^{-1}(u)u^{2\k}du = \infty \textrm{  or equivalently }
\int_{B(0,\varepsilon ) }\tilde \s^{-1}(u)\|u\|^{2 | \k |    }du  = \infty ,
\end{align}
where $  | \k |  = \sum_{i}\k_{i}.$
\end{proposition}

 Let us conclude this subsection with an example illustrating the possible complexity of the concept of $ \k$-pole:
 
   \begin{example}
  Assume that some $ \s:\mathbb{R}^{2}\to\mathbb{R}_{ + }$ satisfies $\s_{}(u)\sim u_{1}^{4}$ as $ u\to 0$.
  Let $ Q$ a non-zero polynomial in $ L^{2}(\s^{-1},B(0,\varepsilon ))$ for some $ \varepsilon >0.$ We put $ Q$ under the form $$ Q(u) =  P(u_{2}) + u_{1}R(u_{2}) + u_{1}^{2}S(u_{1},u_{2})$$ where $ P,R,S$ are polynomials with one or two arguments. We have $ u_{1}^{2}S(u)\in L^{2}(\s^{-1},B(0,1))$, hence so is $ P(u_{2}) + u_{1}R(u_{2})$. By
  isolating dominating terms of $ P$ and $ R$, there are exponents $ p,q\in \mathbb{N}$ and constants $ c_{P},c_{R}$ such that in polar coordinates
  \begin{align*}
\int_{ 0}^{\varepsilon }\left[
\int_{ 0}^{2\pi }\frac{(c_{P} \rho ^{p}\sin(\theta )^{p} + c_{R}\rho \cos(\theta )\rho ^{q}\sin(\theta )^{q})^{2}}{\rho ^{4}\cos(\theta )^{4}}d\theta
\right] \rho d\rho <\infty .
\end{align*}
Examining the point $ \theta  = \pi /2$ in the inner integral, it first implies  $ c_{P} = 0$ , and then $c_{R} = 0$, which means $ P = R = 0$.

  We indeed proved that for a polynomial $ Q = \sum_{}a_{\m}u^{\m}$ in $ L^{2}(\s^{-1},B(0,\varepsilon ))$, all terms of exponent $ \m = (0,m),m\in \mathbb{N}$ or $\m =  (1,m),m\in \mathbb{N}$ must vanish. This means that $ 0$ is a pole  of order $ (1,m)$ for all $ m\in \mathbb{N}$ (and a fortiori of order $ (0,m)$), and it does not admit other orders.
    \end{example}

 \subsection{Comparison with the literature}  

\begin{itemize}
\item 
 Let us first illustrate this result by deriving the two seminal results about number rigidity by  
 
  \begin{example}
 \label{ex:GP}
Let $ \P$ be either $ \P_{\textrm{Gin}}$, the infinite Ginibre ensemble or $ \P_{\textrm{GAF}}$, the zero set of the planar Gaussian analytic function (see  \cite{GP17} for precise definitions). $ \P$ is an isotropic process in dimension $ d = 2$ where  $ \C$ has a density decaying exponentially fast, hence 
\begin{align*}
\s_{}(u) = \s_{}(0) +\frac{ 1}{2} \|u\|^{2}\int_{ }\|t\|^{2}\C(dt) + O(\|u\|^{4}).
\end{align*}
 \cite{GP17} start from the facts that for a twice differentiable function  $ f\in \C_{c}^{b}(\mathbb{R}^{2})$, for some finite $ C,$
\begin{align*}
\var(I_{\P_{\textrm{Gin}}}(f))\leqslant &C\int_{ } \| \nabla f \| ^{2},\\
\var(I_{\P_{\textrm{GAF}}}(f))\leqslant & C\int_{ }| \Delta f|^{2}.
\end{align*}
Since $ \var(I_{\P}(f)) = \int_{ } | \hat f(u) | ^{2}\s(u)du$, this readily implies that in both cases $ \s_{}(0) = 0$ (hyperuniformity), hence by Theorem \ref{thm:main-krigid-iso} for $ \k = 0,$ $ \s_{}^{-1}$ has a $ 0$-pole in $ 0$ and both processes are number rigid. For $\P =  \P_{\textrm{GAF}}$, it also implies that $ \int_{ }\|t\|^{2}\C(dt) = 0$, hence 
\begin{align*}
\int_{ }\|u\|^{2}\s_{}^{-1}(u)du = \infty ,
\end{align*}
and $ \s_{}$ is isotropic. Using also Proposition  \ref{prop:k-incomp-iso} it has a pole of order $ 1$ in $ 0$,
hence $ \P_{\textrm{GAF}}$ is 1-rigid (also proved in  \cite{GP17}).

\end{example} 
 Up to date, $ \P_{\textrm{GAF}}$ and its ``toy model'' (\cite{ST1}) are to the author's knowledge  the only ``disordered'' stationary point  processes for which it is rigourously proven that $ \s_{}(u) = O(\|u\|^{4})$, finding higher order rigidity is a fascinating challenge. We give at Section  \ref{sec:flower} examples for which $ \s(u) \sim c\|u\|^{2n}$ for  $ n$ a prime number  which are not maximally rigid, but is in spirit close to a lattice. The process $ \P_{\textrm{Gin}}$ is an instance of the class of determinantal point processes, the general case is treated at Section \ref{sec:dpp}.\\
 
 \item 
For $ k = 0$, i.e. for number rigidity, our sufficient condition is
\begin{align}
\label{eq:buf}
\int_{B(0,\varepsilon ) }\frac{ 1}{\s_{}(u)}du = \infty .
\end{align}

Let us explain why it allows to recover the sufficient condition of \cite{GL-sufficient} in dimensions $ d = 1,2$. In dimension 1, they prove that a unit intensity point process $ \P$ is number rigid if the truncated correlation measure $ \rho ^{(2)}_{tr}$ has a density   $  \mathscr  c = \F(\s - 1)$ satisfying $ \int_{ } \mathscr  c = -1$ (hyperuniformity) and
\begin{align*}
 |  \mathscr  c(t)   | \leqslant c(1 +   | t | )^{-2},t\in \mathbb{R}.
\end{align*}
It  implies that $\s$ is Lipschitz in $ 0$   :
\begin{align*}
2 |\s (u) |  =   | \s (u) + \s (-u) |   
 = &\int_{ }(e^{itu} + e^{-itu})  \mathscr  c(t)dt + 2\int_{ }  \mathscr  c\\
\leqslant &\int_{ \mathbb{R}} | e^{itu} + e^{-itu}-2 |    |  \mathscr  c(t) |  dt\\
\leqslant & \int_{ }4c\frac{\sin(tu/2)^{2}}{(1 +  | t |)^{2} }dt\\
 = &4c | u | \int_{ }\frac{ \sin(z/2)^{2}}{( | z |  +  | u | )^{2}}dz\\
 \leqslant &C | u | \textrm{  with  }C= 4c\int_{ }\frac{ \sin(z/2)^{2}}{ | z | ^{2}}dz<\infty .
\end{align*}
Hence   \eqref{eq:buf} is satisfied.  In dimension $ 2$, with similar definitions, the condition is 
\begin{align*}
 |  \mathscr  c(t) | \leqslant c(1 + \|t\|)^{-4-\varepsilon }
\end{align*}
which gives that $ \|t\|^{2} \mathscr  c(t)\in L^{1}(\mathbb{R}^{2})$, hence $\s: =  \hat c + 1$ is twice differentiable in $ 0$, and by symmetry $ \s_{}'(0) = 0$, meaning $ \s_{}(u)\leqslant c\|u\|^{2}$, and  \eqref{eq:buf} holds again.\\

\item
 \cite{BufLinear} also give a sufficient condition, but depending on the structure factor of the discretised version, which is not the same as the
original one, see the discussion at Section \ref{sec:discrete}.\\

\end{itemize}
  \subsection{Necessary conditions}

We investigate what non-rigidity implies on the spectral density $ \s$, and in particular wether a $ \k$-pole is necessary for $ \k$-rigidity, it seems to be the first result of this kind. Maximal rigidity  really depends on the behaviour of $ \s$ on all $ \mathbb{R}^{d}$, and is rather studied in  \cite{rigid-companion}. We nevertheless give a context where maximal rigidity and $\k$-rigidity are easier to handle.

 Say that  $ u_{0}\in \mathbb{R}^{d}$ is a pole with finite order for $ \s^{-1}$ if   for some $ \varepsilon >0,q\in \mathbb{N},$
\begin{align*}\int_{ B(u_{0},\varepsilon )}\frac{ \|u-u_{0}\|^{2q}}{\s(u)}du<\infty .
\end{align*} 
 \begin{definition}
 Say that $ \s$ is simple if  $ \s_{}^{-1}$ has finitely many poles, all with finite order, and $ \s_{}(u)\geqslant c(1 + \|u\|)^{-p}$   for some $ p\in \mathbb{N},c>0$, at some distance $ \varepsilon>0 $ from these poles. \end{definition}  

\begin{proposition}
\label{prop:not-lmr-finite-z}
Assume that $ \s_{}$ is  {\it simple}. Then $ \M$ is not linearly maximally rigid on $ B(0,\eta )$ for $ \eta >0$. Furthermore, for $\k \in \mathbb{N}^{d},$ $ \M$ is linearly  $\k$-rigid on $ B(0,\eta )$  if and only if  $ 0$ is a pole of order  $\k$. 

\end{proposition}

The proof is at  Section  \ref{sec:prf-converse}. This proposition in particular serves to treat stationary Determinantal Point Processes  as their structure factor vanishes only in $ 0$, see Section \ref{sec:dpp}.

Apart from the  {\it simple} case, we can  give converse statements under stronger structural assumptions: for $ \s_{}$  {\it isotropic} (i.e. $ \s_{}(u)$ only depends on $ \|u\|$) or  {\it separable} (i.e. $ \s_{}(u) = \s_{1}(u_{1})\dots \s_{d}(u_{d}),u\in \mathbb{R}^{d}$ for some even functions $ \s_{i}:\mathbb{R}\to \mathbb{R}_{ + }$) and by assuming upfront that maximal rigidity does not hold, which often can be proved by other means.  Remark  \ref{prop:rigid-monotone} yields that if rigidity does not hold for some $ \s'\leqslant \s_{}$, then it does not hold either for $ \s_{}$, hence in the general case one can investigate the largest separable/isotropic function $ \s'\leqslant \s_{}.$

 \begin{theorem}
  \label{thm:converse-rigidity}
Let $ \M_{}$ with spectral density $ \s$ with either  \begin{itemize}
\item $ \s_{}$ is isotropic and $ A = B(0,R),R>0$
\item  $ \s_{}$ is separable and $A =  [-R,R]^{d},R>0$. \end{itemize}
Assume $ \M$ is not maximally rigid on $ A$ and   \eqref{eq:ass-general-k-rigid} does not hold.
 Then for $ \eta >0$, $ \M$ is not $ {\k}$-rigid on $ A^{ + \eta }$.

  \end{theorem}  

Proof at Section  \ref{sec:prf-cvs-main}  .

\begin{remark}
Under these hypotheses, if it can be established by other means that $ \M$ is linearly 0-rigid or 1-rigid, one can reverse the reasoning of Example  \ref{ex:GP} and deduce bounds on the variance of linear statistics similar to those of $ \P_{\textrm{GAF}}$ and $ \P_{\textrm{Gin}}$.
\end{remark}

In view of Proposition  \ref{prop:not-lmr-finite-z}, Theorems \ref{thm:main-krigid-iso} and  \ref{thm:converse-rigidity},  \eqref{eq:ass-general-k-rigid} is necessary and sufficient for $ \k$-rigidity on $ A$ in the following cases \begin{itemize}
\item   $ \s_{}$ is ``simple'', 
\item $ \M$ is not LMR and $ \s_{}$ is quasi-isotropic, i.e. there is $ c_{-},c_{ + }>0$ such that 
\begin{align*}
c_{-}\s_{}(\|u\|)\leqslant \s_{}(u)\leqslant c_{ + }\s_{}(\|u\|),u\in \mathbb{R}^{d},
\end{align*}
\item $ \M$ is not LMR  and $ \s_{}$ is quasi-separable, i.e. $ 0\leqslant c_{-}\s'\leqslant \s_{}\leqslant c_{ + }\s' $ for some separable $ \s'$.
\end{itemize} 

 \begin{corollary}[No phase transition]
 For $ \M$ which is not LMR on any convex body, linear number rigidity  happens either on every  convex body or none. 
     \end{corollary}
%
  On the other hand, we give in  \cite{rigid-companion} examples of Gaussian measures which are maximally rigid for  $ R\leqslant R_{c}$ for some $ R_{c}>0$, and not even $ 0$-rigid for $ R>R_{c}$.
  There are examples of processes obtained as  (perturbations of) hyperplane networks, such as in  \cite{KLY,PS14}, where number rigidity occurs in dimension $ d = 2,3,...$, but it does not enter in the current framework as it can be proved that this  rigidity is non-linear, and the tools employed strongly depend on the underlying network structure. We are not aware of models which are rigid but not linearly rigid and not derived from a network construction. We think that for some continuous random fields, the set of zeros is rigid but not linearly, this will be the object of future research.\\

  The structural conditions on $ \s$ are necessary as there exists random measures which are linearly $ \k$-rigid without $ \s$ vanishing around $ 0.$ We give an example for mass rigidity.
   
 \begin{proposition}\label{ex:spectral-counterexample}
 Any WSS random measure  on $ \mathbb{R}^{2}$ which spectral density satisfies $ \s(u) \leqslant c u_{2}^{2}$  in some neighbourhood of $ (1,0)$ is  $ 0$-rigid on any convex body $ A$.
 \end{proposition} 
 Proof at Section  \ref{sec:prf-counterex}.
 We can therefore build examples which are completely standard, i.e. with fast decay of the correlations, linear variance, and which are rigid; by the previous results, such examples will never be isotropic or separable. To give a concrete example, take a mixture of an isotropic and a separable density
\begin{align*}
\s(u) = 
\begin{cases} 
1$  if $\|u\|\leqslant 1/2\\
 \frac{  u_{2}^{2}}{1 + u_{2}^{10}} \frac{ 1}{1 + u_{1}^{10}}$  otherwise$, \end{cases}
\end{align*} not LMR because $\psi (u) :=  u_{2}(1 + \|u\|^{100})^{-1}\in L^{2}(\s^{-1}).$ To  {\it realise} $ \s$, i.e. find $ \M$ for which such $ \s$ is indeed the spectral density, use for instance Gaussian fields with Example \ref{ex:gauss}. Whether there are such number rigid point processes with standard second order behaviour which are not hyperuniform  depends on our ability to have such spectral densities for random atomic measures, this relates to the generally difficult problem of  {\it realisability of point processes}.

%
%
%
%

 \subsection{$\k$-tolerance}
 
 \cite{GP17} show that  the infinite Ginibre ensemble $ \P_{\textrm{Gin}}$ is not rigid at an order higher than $ 0$, in the sense that if one determines $ N : = \P_{\textrm{Gin}}(A)$ from $ \P_{\textrm{Gin}}1_{A^{c}}$, then conditionnally to $ \P_{\textrm{Gin}}1_{A^{c}}$ and $ N$, $ \P_{\textrm{Gin}}1_{A}$'s law is continuous wrt $  {\bf U_{N}} = \{U_{1},\dots ,U_{N}\}$, consisting of i.i.d.   uniform variables  on $ A.$   They similarly show that the zero set $ \P_{\textrm{GAF}}$ of the planar GAF is $ 1$-rigid and not more, in the sense that conditionnally to $$ (N,U,\Phi ): = (   \P_{\textrm{GAF}}(A),  \int_{ }xd\P_{\textrm{GAF}}(x), \P_{\textrm{GAF}}1_{ A^{c}}),$$ $ \P_{\textrm{GAF}}1_{A}$ has a continuous law wrt to the law of $  {\bf U}_{N}$ conditionned to $ \sum_{i = 1}^{N}U_{i} = U.$ A similar result is proved by  \cite{DHLM} for the sine$ _{\beta }$ processes $ \P_{\beta },\beta >0$.
 
 These strong distributional results rely on the determinantal nature of $ \P_{\textrm{Gin}}$, the Gaussian nature of $ \P_{\textrm{GAF}}$, or the DLR equations defining $ \P_{\beta }$, and cannot be extended in the current general context.
We can still investigate how ``linearly tolerant'' is a process which is not rigid, i.e. whether the random moments $ m_{\m}(A): = \int_{A }t^{\m}d\M(t),\m\in \mathbb{N}^{d}$ are tied by a finite linear relation 
\begin{align*}
\sum_{\m}a_{\m}m_{\m}(A) \in \sigma _{lin}(\M_{A^{c}})
\end{align*}
or equivalently if we have $ Q$-rigidity on $ A$ for the polynomial $ Q(t) =\sum_{\m}a_{\m}t^{\m} $. We show that this holds for what we call a  {\it $ \k$-polynomial} $ Q$, i.e. such that for some $ \k_{0}\succeq \k$, $ a_{\k_{0}}\neq 0$ and the non-vanishing terms $ a_{\m}\neq 0$ are such that $ \m\preceq \k_{0}.$

  \begin{proposition}
  \label{thm:tolerant-simple}
   Assume that $ 0$ is not a pole of order $ \k$ and  the assumptions of Theorem  \ref{thm:converse-rigidity} or of Proposition \ref{prop:not-lmr-finite-z} hold. Then $ \M$ is not $ Q$-rigid for any $ \k$-polynomial $ Q$. \end{proposition}  
   
The proof of this proposition is actually included in the proofs of resp.   Proposition \ref{prop:not-lmr-finite-z} and Theorem \ref{thm:converse-rigidity}.
As is apparent from the proofs, we will in fact have no $ Q$-rigidity for almost all polynomials having at least a non-vanishing term of order $ \k_{0}\succeq \k$, i.e. for a linear combination $ Q = \sum_{i}\alpha _{i}Q_{i}$ where the $ Q_{i}$ are $ \k$-polynomials, but some particular combinations of $ \alpha _{i}$ are harder to discard.

\subsection{Point processes with arbitrary rigidity}
\label{sec:flower}

For systems which have a non-zero spectral measure, the highest order of  rigidity  that is encountered in popular models is the $ 1$-rigidity of the zeros of the planar GAF, see  \cite{GP17}, it corresponds to a structure factor decaying in $ 0$ as $ \|u\|^{4}$ in dimension $ d = 2$ (see Example  \ref{ex:GP}).  \cite{ST1} provide a toy model with a simpler structure that exhibits similar properties, in particular for the rigidity order and decay of the structure factor.  \cite{torquato-tilings} give a heuristic for building point processes with structure factor decaying arbitrarily fast. 

We give here   in dimension $ 2$ a class of simple models presenting a arbitrarily high order of rigidity. 
Let $ \rho >0$ and $ n\in \mathbb{N}^{*}$. Define $ \P_{\rho ,n}$ as the process obtained by putting at distance $ \rho $ of each point $ \m$ of $ \mathbb{Z} ^{d}$ the $ n$-th order roots of unity rotated by a random independant quantity $ \theta _{\m}$, and applying a global stationarising shift $ U$:\begin{align*}
\P_{\rho  ,n} = \sum_{\m\in \mathbb{Z} ^{d}}\sum_{z\in  {\bf U}_{n}}\delta _{U + \m + \rho e^{i\theta _{\m}}z}
\end{align*}
where $  {\bf U}_{n} = \{e^{2\pi \imath k/n};k = 1,\dots ,n\}\subset  \mathbb C $, $ U$ is an independent uniform random variable in $ [0,1]^{2}$, $ \theta _{\m}$ are i.i.d.   uniform numbers on $ [0,2\pi ]$. The structure factor $ \S$ is defined as before by assimilating $  \mathbb C $ to $ \mathbb{R}^{2}$.

\begin{theorem} \label{thm:flower}
 For $ n\geqslant 3$ a prime number,    $$ \S(u)\sim \frac{\rho ^{2n}}{(n-1)!^{2}}u^{2n}$$ as $ u\to 0,u\in  \mathbb C .$ Therefore, $ \P_{\rho   ,n}$ is $ (n-1)$-rigid on any convex body $ A\subset  \mathbb C $.
 \end{theorem}  
 
 Most likely, if $ n$ is not prime, the order of decay will be determined by the smallest prime divisor of $ n.$
 Any other base model with a spectral gap instead of $ \mathbb{Z} ^{d} + U$ (i.e. a  {\it stealthy process}) would give the same results.
 
For the proof, we establish a general formula for the structure factor of a unit intensity stationary point process $ \Phi \subset \mathbb{R}^{d}$ perturbed by an arbitrarily cluster of points. 
 Let $ \c  _{0} = \sum_{i = 1}^{N}\delta _{X_{i}}$ a random finite point process with law $ \mu $, where $ N\in \mathbb{N}$ and the $ X_{i}\in \mathbb{R}^{d},i\geqslant 1,$ are random.   The average number of points is denoted by $ \kappa  = \mathbf{E}(N)$.
Let $ \c  _{\m},\m\in \Phi $ i.i.d.    copies of $ \c  _{0}$ (cond. to $ \Phi $) and
\begin{align*}
\Phi _{  \mu  }  :=\sum _{\m\in \Phi  }\tau   _{\m} \c  _{\m}
\end{align*}
where  points are counted according to multiplicity.
We say that $ \Phi $ is pertubed by a cluster with  {law} $ \mu  $.
Remark that $ \Phi_{\mu }   $ has intensity $$\mathbf{E}(\Phi_{\mu }(A)) =   \kappa \Leb(A),A\subset \mathbb{R}^{d},$$ with possibly multiple points. 

\begin{proposition} 
\label{lm:flower}
Let $ \Phi $ a $ L^{2}$ wide sense stationary  unit intensity point process with structure factor $ \S_{\Phi }$ and $ \mu $ such that $ \mathbf{E}(\#\c_{0}^{2})<\infty .$
Let $ \varphi (u)= \F\c_{0} (u)= \int_{ }e^{-\imath \langle u,t \rangle}\c_{0}(dt)$. Then the structure factor of $ \Phi_{\mu }$ is 
\begin{align*}
\S(du) = ( \mathbf{E} | \varphi (u) | ^{2}- | \mathbf{E}\varphi (u) | ^{2})\Leb(du) +  |\mathbf{E}\varphi (u) | ^{2}\S_{\Phi } (du),u\in  \mathbb C .
\end{align*}
\end{proposition}

This result, of independent interest, is proved at Section  \ref{sec:flower-prf}. It generalises and strenghtens the result for lattices  perturbed by i.i.d.   points when a.s. $ \c_{1}(\mathbb{R}^{d}) = 1$  (\cite[Th. 9.3]{Coste}). 

 \begin{proof}[Proof of Theorem  \ref{thm:flower}]
 
In the case of $ \P_{\rho   ,n}$, $$ \Phi  = \sum_{\m\in \mathbb{Z} ^{d}}\delta _{U + \m},\;\; \c_{0} = \sum_{k = 1}^{n}\delta _{\rho e^{\imath k \alpha  + \theta }}$$ where $\alpha  = 2\pi /n$  and $ \theta $ is a random element of $  [0,2\pi ]  $.  
We define for $ u\in  \mathbb C,\theta \in [0,2\pi ] $, 
\begin{align*}\varphi _{\theta }(u) = &\sum_{k = 1}^{n}\exp(-i \langle e^{i(k\alpha  + \theta )},u \rangle),
\end{align*}
so that $ \varphi(u)  = \varphi _{\theta }(\rho u)$ where $ \theta $ is random. Since the law of $ \c_{0}$ is invariant under rotations it is enough to compte $ \varphi (u)$ for $ u\in \mathbb{R}.$
We use the Taylor expansion with the Lagrange form of the remainder: for $ u\in \mathbb{R}$  
\begin{align*}
\varphi _{\theta }(u) = \sum_{p = 0}^{2n }\frac{ (-iu)^{p}}{p!}  \sum_{k = 1}^{n} \cos(k\alpha + \theta) ^{p} + u^{2n + 1}R_{n}(u,\theta )
\end{align*}where $R_{n}(u,\theta ) = O(\|\varphi _{\theta }^{(2n + 1)}\|_{\mathbb{R}}) $   is uniformly bounded over $ u,\theta \in \mathbb{R}$. We also use the fact that for some real coefficients $ \lambda _{m,p}$ with $ \lambda _{p,p} = 1$,  for $ \gamma \in \mathbb{R},$
\begin{align*}
 \cos(\gamma  )^{p} =&  \sum_{m = 0}^{p}\lambda _{m,p}\cos(m\gamma )\\
 \sum_{k = 1}^{n}\cos(k\alpha  + \theta )^{p}  = &  \sum_{m = 0}^{p}\lambda _{m,p}   \sum_{k = 1}^{n}\cos(m(k\alpha  + \theta )).
\end{align*}For $ p\in \{1,\dots ,2n\}$,
the terms $ m \notin \{0,n,2n\}$ vanish because $$\sum_{k = 1}^{n}\cos(m(k\alpha  + \theta )) =  \mathscr  R \sum_{k = 1}^{n} e^{imk\alpha }e^{im\theta } =   \mathscr  R\left(
e^{im\theta }    \sum_{k = 1}^{n}e^{imk\alpha } 
\right)= 0,$$ because $ m$ is prime with $ n$. We have  for $ p\leqslant 2n$
\begin{align*}
\sum_{k = 1}^{n}\cos(k\alpha  + \theta )^{p} = &  n\lambda _{0,p} + {\mathbf 1}\left[ n\leqslant p  \right] n\lambda _{n,p}\cos(n\theta ) + \mathbf{1}_{\{p = 2n\}}n\cos(2n\theta ).
\end{align*}
 Finally,
\begin{align*}
 \varphi _{\theta }(u) = \sum_{p = 0 }^{2n}\underbrace{\frac{ (-i)^{p}}{p!}n\lambda _{0,p}}_{c_{p}}u^{p} +\cos(n\theta )\sum_{p = n}^{2n}\underbrace{\frac{ (-i)^{p}}{p!}n\lambda _{n,p}}_{a_{p}}u^{p}  + \frac{(-iu)^{2n}}{(2n)!}n\cos(2n\theta )+ u^{2n + 1}R_{n}(u,\theta ).
\end{align*} 
From that we can deduce the expansion of $ |  \varphi _{\theta }(u) | ^{2} = \sum_{l = 0}^{2n}\gamma _{l}(\theta )u^{l} + u^{2n + 1}S_{n}(u,\theta )$, where $ S_{n}(u,\theta )$ is uniformly bounded for the same reasons. Hence we have\begin{samepage}
\begin{align*}
 \mathbf{E} | \varphi  (u) | ^{2}=&  \sum_{\substack{p,q = 0\\p + q\leqslant  2n}}^{2n}c_{p}  \bar c_{q} (\rho u)^{p + q} 
 +   \mathbf{E}(\cos(n\theta )^{2})(\rho u)^{2n} | a_{n} | ^{2}\\
 &+ \mathbf{E}(\cos(n\theta ))\sum_{p = 0}^{2n}\sum_{q =n }^{2n-p}2  \mathscr  Rc_{p}  \bar a_{ q}   \rho ^{ p + q}u^{ p + q}  +\mathbf E (\cos(2n\theta ))2c_{0} \frac{(-iu)^{2n}}{(2n)!}n+ (\rho u)^{2n + 1}S_{n}(\rho u,\theta )
\\ | \mathbf{E}\varphi (u) | ^{2}  = & \left|
\sum_{p = 0 }^{2n}c_{p}   \rho ^{p  }u^{p }  +  (\rho u)^{2n  + 1}\mathbf{E}(R_{n}(\rho u,\theta ))
\right|^{2} =  \sum_{\substack{p,q = 0\\p + q\leqslant  2n}}^{2n}c_{p}  \bar c_{q} \rho ^{p + q}u^{p + q} 
 + O(u^{2n + 1})
\end{align*}\end{samepage}
using $ \mathbf{E}(\cos(n\theta )) =\mathbf E (\cos(2n\theta )) =  0$,
whence indeed, as claimed,
 $$ \mathbf{E} | \varphi  (u) | ^{2}- | \mathbf{E}\varphi (u) | ^{2} =  \frac{n^{2}}{n!^{2}}(\rho u)^{2n}   + O(u^{2n + 1}),$$ and this expansion is also valid on $  \mathbb C .$ The conclusion comes from the fact that  $ \S_{\Phi } = \sum_{\m\in \mathbb{Z} ^{d} \setminus \{0\}}\delta _{\m}$ (\cite{Coste}).
The $ (n-1)$-rigidity   is then a straightforward consequence of  Proposition \ref{prop:k-incomp-iso}. 
%
 \end{proof}

 \begin{remark}[Higher dimensions]
 
 It is possible to lift this example to higher dimensions in a rather artificial way: attach to each point of $ \mathbb{Z} ^{d} + U$ a cluster   which points are the $  (u_{1},u_{2},0,\dots ,0)$  where $ (u_{1},u_{2})$ is an  atom of an independent version of $ \c_{0}$.   \cite{torquato-tilings} give constructions yielding high order decay in dimension $ 3.$
 

 \end{remark}

\subsection{Rigidity for discrete processes}
\label{sec:discrete}

Let us  derive similar results for a centred discrete process $ \X = \{\X_{\m};\m\in \mathbb{Z} ^{d}\}$ with the linear statistics \begin{align*}
\X(\gamma ) := \sum_{\m}\X(\m)\gamma (\m).
\end{align*}
%
We assume
$ \X$ is wide-sense $ L^{2}$-stationary with spectral measure $ \S$, i.e. for $ \gamma $ with bounded support
\begin{align*}
 \textrm{Var}\left(\X(\gamma )\right) = (2\pi )^{-d}\int_{  \mathbb  T  ^{d}}  | \hat \gamma  (u)| ^{2}\S(du)
\end{align*}
where $  \mathbb  T   = [-\pi ,\pi ]$,
for the complex-valued  trigonometric polynomial  
\begin{align*}
\hat \gamma (u) = \sum_{\m\in \mathbb{Z} ^{d}}e^{-\imath \langle \m, u\rangle}\gamma (\m).
\end{align*}
For $ A\subset \mathbb{Z} ^{d}$, the class of such trigonometric polynomials where $ \gamma (\m) = 0$ for $\m\notin A$  is denoted by $ \E( A)$, by analogy with entire functions of exponential type (see Section \ref{sec:theory}).
\begin{theorem}~
\label{thm:discrete}
Let   $ \s(u)du$ the continuous component of $ \S$.
\begin{itemize}
\item 
$ \X$ is not LMR on $ \llbracket m \rrbracket^d $  if and only if  there is  $ \psi\in \E(\llbracket m \rrbracket^d ) $ such that 
\begin{align}
\label{eq:lmr-poly-discrete}
\int_{ }\frac{ |  \psi (u) | ^{2}}{\s_{}(u)}du<\infty .
\end{align}
\item For $ \k\in \mathbb{N}^{d}$, $ \X$ is not $ \k$-rigid  if and only if  there is  $\psi \in \E( \llbracket m \rrbracket^d) \cap L^{2}(\s_{}^{-1})$ such that  $ \partial ^{\k}\psi (0)\neq 0$, where $ \partial ^{\k} = \partial _1^{\k_{1}}\dots \partial _{d}^{\k_{d}}.$

\item  If $ 0$ is a pole of order $ \k$ for $ \s_{}^{-1}$, then $ \X$ is $ \k$-rigid.

\end{itemize}

 \end{theorem}

    \begin{proposition}    [Specialization to time series in dimension $ d = 1$]
    \label{prop:spec-1D}  $ \X$ is not LMR on $ \llbracket m \rrbracket$  if and only if  the number of poles of $ \s_{}^{-1}$ on $ [-\pi ,\pi )$ counted with multiplicity is $\leqslant 2m$.
Also, $ \X$  is  not $ 0$-rigid   if and only if   there are $ 2m$ poles or less and $ 0$ is not a pole.
    \end{proposition}
 
The proof is at Section \ref{sec:discrete-prf}. We discuss the LMR result below,  the other statements  seem to be the first  dealing with partial rigidity, i.e. when one is concerned not in retrieving all of $ \X$ on $ \llbracket m \rrbracket^d $, but only its restricted moments, by analogy with the continuous case.

 The line of research of linear interpolation in the 20th century was mainly concerned with time series, and authors noted that for a polynomial satisfying \eqref{eq:lmr-poly-discrete}, perfect linear interpolation is impossible,  see  (10.28) and Theorem 10.3 in  \cite{Rozanov}, with also an extension to processes taking values in $ \mathbb{R}^{q}$.
   Much later,  \cite{LyonsSteif} study the related problem of the strong full K property for $ \{0,1\}$-valued $ \X$, i.e. what does it imply on the $ \X(\m), | \m | \leqslant m$, if we impose $ \X(\m) = 1$ for $ m\leqslant  | \m | < m + k$ for $ m,k$ arbitrarily large? It is clear indeed that if no polynomial $ \psi $ is in $ L^{2}(\s^{-1})$, then $ \X$ is LMR for all $ m$, and as such for $ k$ large enough the external conditioning completely determines $ \X(\m), | \m | \leqslant m$. This strong resemblance explains why condition (iii) in their theorem 7.7 is exactly the same condition. None of these works seems to acknowledge that  $ m$ must be  the degree of $ \psi .$ This result also completes the answer of  \cite{GhoshComplete} to the Question 9.9 of  \cite{LyonsSteif} regarding rigidity on $ A = \{0\}$.
    
     \cite{BufLinear} studies the rigidity of continuous models through a discretisation procedure.  Given a stationary random measure $ \M$ on $ \mathbb{R}^{d}$, one can consider the associated stationary discrete process 
\begin{align*}
\X_{\m} = \M(\m + [-1/2,1/2]^{d}),\m\in \mathbb{Z} ^{d}.
\end{align*} They show number rigidity of $ \M$ on $ [-1/2,1/2]^{d}$ by showing the rigidity of $ \X$ on $ \{0\}$, and derive a necessary condition in terms of the covariance decay for $ \X.$

%

 \begin{remark}
 A reformulation of $ 0$-rigidity in dimension $ 1$ is the following: $ \X$ is $ 0$-rigid if  $ \X$ is $ LMR$ or $ 0$ is a pole, which means that either $ \int_{ B(0,\varepsilon )}\s_{}^{-1} = \infty $ or that the poles are away from zero and strictly more than $ m$ (counted with multiplicity), i.e. if one considers only $ m$ points, necessarily at least  one pole is not covered, i.e. for $ u_{1},\dots ,u_{m}\in   \mathbb C ^{d}  \setminus \{1\},$
\begin{align*}
\int_{ }\frac{ \prod_{i} | e ^{iu}-e^{iu_{i}} | ^{2}}{\s_{}(u)}du = \infty .
\end{align*}
This seems to correspond to the content of Remark 2.1 of  \cite{BufLinear} for continuous measures $ \S = \s_{}\Leb$.x
 \end{remark}

At the contrary of the continuous case, $ k$-rigidity is not monotonous in $ k$, in the sense that there are examples of $ 1$-rigid processes that are not $ 0$-rigid in dimension $ 1$, due to the constraint on the degree:
\begin{example}
\label{ex:discrete}
Consider $ m = 1$ in dimension $ d = 1,$ and $ \s(u) = (u-1)^{2}(u + 1)^{2},u\in \mathbb{T}$. 
Let $ \psi $ a degree $ 1$ trigonometric polynomial of $ L^{2}(\s^{-1})$. In particular, there is a polynomial $ P$ of degree $ 2$ such that 
\begin{align*}e^{iu}\psi (u) = P(e^{iu}) = a(e^{iu}-z_{1})(e^{iu}-z_{2}),u\in  \mathbb  T  ,
\end{align*}for some $ a\in  \mathbb C  \setminus \{0\},z_{1},z_{2}\in  \mathbb C .$ Examining the neighbourhoods of $ 1$ and $ -1$ of $ \s^{-1}$, we necessarily have $ z_{1} = e^{i},z_{2} = e^{-i}$ (or the other way around), this is also sufficient for $ \psi \in L^{2}(\s^{-1})$, showing that $ L^{2}(\s^{-1})\cap \E(\llbracket 1 \rrbracket)\neq \{0\}$, hence $ \s$ is not LMR. It also implies that every such $ \psi $ is symmetric, hence $ \psi '(0) = 0$, meaning that $ \s$ is $ 1$-rigid. On the other hand, $ P(0)\neq 0$, hence $ \s$ is not $ 0$-rigid. 
\end{example}

 \section{Applications}   
 \label{sec:appli}

 \subsection{Determinantal processes and completeness
} 
\label{sec:dpp}

Determinantal points processes (DPPs), introduced in the context of quantum mechanics, have gained popularity as many classes of essential models in Random Matrices, Statistical Physics, Combinatorics and other fields have proven to be determinantal, see \cite{BKPV}. In the Euclidean context, we follow  \cite{soshnikov}: a simple   point process $ \P $ on $ \mathbb{R}^{d}$ is determinantal with  kernel 
 $ K:\mathbb{R}^{d}\times \mathbb{R}^{d} \to \mathbb{C}$ if for every $ k\in \mathbb{N}^{*}$, 
\begin{align*}\rho _{k}(x_{ 1},\dots ,x_{k}): = \det((K(x_{i},x_{j}))_{ 1\leqslant i,j\leqslant k})\geqslant 0
\end{align*}
is the $ k$ point correlation function of $ \P $, i.e. 
for any non-negative $ \varphi :(\mathbb{R}^{d})^{k}\to \mathbb{R}$, we have 
\begin{align*}\mathbf{E}\left(
\sum_{x_{ 1},\dots ,x_{k}\in \P }^{\neq }\varphi (x_{ 1},\dots ,x_{k})
\right) = \int_{}\rho _{k}\varphi ,
\end{align*}
where the sum runs over $ k$-tuples of pairwise distinct points. We assume up to a scalar rescaling that $ K(0,0) = 1$ to deal with unit intensity processes. Stationarity yields that  $ \rho _{2}(x,y) = \rho _{2}(x-y)$ with an abuse of notation. Not all functions $ K$ give rise to a DPP and in particular we will require that $K$ is Hermitian and positive definite, so that 
\begin{align*}\rho _{2}(x-y) = 1- \kappa(x-y)^{2}\in \mathbb{R}_{ + }
\end{align*}  where $ \kappa(x) =    | K(0,x)  | , 0\leqslant  \kappa(x)  \leqslant 1$, we also have to assume that $ \kappa $ is locally square integrable for definiteness issues. See  \cite[Section 4.5]{BKPV} for unicity and existence questions, and  \cite{soshnikov} for a treatment of such stationary DPPs, it is shown in particular that they are mixing and ergodic, and the trace-class property yields that $ \hat \kappa $ is integrable and $ 0\leqslant \hat \kappa \leqslant 1$.

There have been a lot of works on the rigidity (mainly of order $ 0$) of continuous DPPs  \cite{GP17,GhoshComplete,BufAiry,BufBalanced, BufLinear,Buf-conditional}.
The main result of this section is a characterisation of $ \k$-rigidity in the continuous setting:
\begin{corollary}
\label{thm:k-rigid-dpp}Let $ \P$ be a stationary DPP on $ \mathbb{R}^{d}$ with Hermitian positive kernel $ K$ and $ \kappa(x) =  | K(0,x) | $  square integrable.
 Then for $ \k\in \mathbb{N}^{d},$ $ \P$ is $ \k$-rigid on a convex body $ A$   if and only if  $ (1- \widehat {\kappa ^{2}})^{-1}$ has a pole of order $ \k$ in $ 0$. When it is not the case, $ \P$ is $ \k$-tolerant in the sense that it is not $ Q$-rigid on $ A$ for any $ \k$-polynomial $ Q$.
 \end{corollary}  
 
We see that $ \s: = (1- \widehat {\kappa ^{2}}) $ must vanish in $ 0$ to have some rigiditiy. That is because, by  \cite{soshnikov}, $ \s(u)du$ is indeed the structur factor of $ \P$, and $ \s(0) = 0$ exactly means that $ \P$ is hyperuniform (see \eqref{eq:hu}); hyperuniform DPPs are a very import class in Statistical Physics and Random Matrices, and are also used for tasks of numerical integration  \cite{BH16} or machine learning   \cite{KT12}.
 
 \begin{proof} 
Since $\s(u)\geqslant 0, \int_{ }\kappa ^{2}\leqslant 1$, and $ \s(u) = 0$ is only possible if $ u = 0$. 
 Since $ \kappa ^{2}$ is integrable, 
 the Riemann-Lebesgue lemma ( \cite[Th. 14.2]{DuiKolk}) yields that $  | \widehat {\kappa ^{2}}(u)| <1/2$ for $  | u | \geqslant T$ for some $ T. $ 
 By Theorem \ref{thm:main-krigid-iso}, $ \P$ is $ \k$-rigid on $ A$ if  $ \s_{}^{-1}$ has a pole of order $ \k$ in $ 0$.
If it is not satisfied,   it follows by Proposition  \ref{thm:tolerant-simple} that $ \M$ is not $ Q$-rigid for a $ \k$-polynomial $ Q$.
  
 \end{proof}
Linear number rigidity is therefore equivalent to the non-integrability of $ (1- \widehat {\kappa ^{2}})^{-1}$.
For many examples, $ \widehat {\kappa ^{2}} (u)= 1-\sigma ^{2}u^{2} + o(u^{2})$ for some $ \sigma >0$, hence there is $ \k$-rigidity  if and only if  $ \k = 0$ and $ d\in \{1,2\}$, and no rigidity in higher dimensions. This applies for instance to the infinite Ginibre ensemble on $\mathbb{R}^{2}\approx  \mathbb C $, defined by 
\begin{align*}K(x,y) = e^{x\bar y-\frac{1}{2}|x|^2-\frac{1}{2}|y|^2};x,y\in  \mathbb C ,
\end{align*}i.e. $ \kappa(x) = e^{-\frac{  | x | ^{2}}{2}}$, retrieving the result of  \cite{GP17}, or the result of   \cite{GhoshComplete}, stating that  the sine process in dimension $ 1$ and Ginibre ensemble in dimension $ 2$ are number rigid. Tensor kernels give non rigid examples: \cite{BufLinear}  proves non-rigidity for the  DPP  which kernel is  $ \kappa(x) = \sinc(x_{1})\sinc(x_{2})$ on $ A = [-1/2,1/2]^{2}$ with a discretisation (see Section  \ref{sec:discrete}). Theorem  \ref{thm:k-rigid-dpp} yields that this $ \P$ is not linear number rigid in the continuous sense on any convex body  $ A$, 
 any discrete average   will not be linearly rigid either.\\

 This result also bears a  connection to the question of finding random countable sets of exponential functions spanning $ L^{2}(A)$ for a compact $ A\subset \mathbb{R}^{d}$. For $ \chi \subset \mathbb{R}^{d} $, call 
\begin{align*}E_{\chi } = \{t\mapsto e^{\imath tx};x\in \chi \}.
\end{align*}
Say that $ E_{\chi }$  is  {\it complete}  if  functions of $ L^{2}([-\pi ,\pi ])$ are $ L^{2}$ approximable by functions of $ E_{\chi }$.  \cite{LyonsSteif} investigated the question of completeness   for a random set $ \chi\subset \mathbb{R}^{d} $, and more particularly if $ \chi $'s law is a DPP, leaving open several questions.  \cite{GhoshComplete} established a connection between number rigidity and completeness, and a corollary of  \cite[Th. 1.3]{GhoshComplete} and Corollary  \ref{thm:k-rigid-dpp} is the following result:
\begin{corollary}
If $ \P  $ is a  DPP, then $ E_{\P  }$ spans $ L^{2}([-\pi ,\pi ])$ if $ (1- \widehat {\kappa ^{2}})^{-1}$ is not integrable in $ 0.$
\end{corollary}

  In particular it retrieves Th. 1.5 of  \cite{GhoshComplete}, which shows number rigidity in dimension $ 1$ if $ \hat  \kappa $ is the indicator function of an interval (i.e. $ \hat \kappa  = \hat \kappa ^{2}$): Parseval equality yields
\begin{align*}
\widehat{ \kappa ^{2}}(0) = \int_{ }\kappa ^{2} = \int_{ } \hat \kappa ^{2} = \int_{ } \hat \kappa  = \kappa (0) = 1,
\end{align*} 
hence $ \s(0) = 0$ (meaning $ \P$ is hyperuniform), and $ \widehat{ \kappa ^{2}}$ is clearly Lipschitz, hence $ \s(u) = O( | u | )$ and we have number rigidity.

 \subsection{Covariance of Gibbs measures}
 \label{sec:gibbs}
 
 Gibbs measures are also a prominent class of point processes, especially in statistical physics  (\cite{Lewin}).
 It is known that in the case of short range interactions, such processes are not hyperuniform, see \cite{DereudreFlimmel}, hence the structure factor $ \S$ will not vanish around $ 0$ and the process is most likely not rigid.
On the other hand, those with a very strong dependency at long range, such as Coulomb gases, or more generally Riesz gases, form a very important class of models in statistical physics, and are expected to be hyperuniform,  see \cite[VI.C]{Lewin} and references therein.  

   Such models are usually defined for a finite but large number of particles interacting through some physical equations, and defining an infinite stationary model compatible with these equations is already a challenge, there can in many instances be distinct infinite models compatible with local conditions. In dimension $ 1$,  \cite{ValkoVirag} have reached a universal explicit limit of 1D sine$ _{\beta }$ log gases, further studied in  \cite{DHLM,ChhaibiNaj}, where it is proved that it is number rigid. Not much seems to be known in higher dimensions on the rigidity (or existence) of stationary models. Let us emphasize that hyperuniformity \eqref{eq:hu} can be equivalently defined in a simpler way, namely we only require that 
\begin{align}
 \label{eq:weak-hu}
\frac{ \textrm{Var}\left(\sum_{x\in \P}f(R^{-1}x)\right)}{R^{d}}\to 0
\end{align}
for some Schwartz function $ f$ with $ \int_{\mathbb{R}^{d}}f\neq 0$ (this is usually simpler to show than for $ f$ the unit ball indicator as in \eqref{eq:hu}).
 Summing up the current main results  in this context, $ k$-rigidity of a hyperuniform isotropic (or 1D) model is equivalent to 
\eqref{eq:ass-general-k-rigid},
which is already a rigorous statement on the second-order properties of the process through $ \s$, though sometimes difficult to connect with the properties of the covariance $ \C$ if the correlations are not known to decay fast. In particular, proving hyperuniformity is already a challenge, the only rigorous result in dimension $ d\geqslant 2$ is, to my knowledge, the work of \cite{Leble}, showing that 2D Coulomb gases are hyperuniform.

 In the other direction,  \cite{DHLM} and  \cite{DereudreVasseur} have established non-rigidity results, the latter managed to prove the existence of an infinite stationary and isotropic model in all dimension $ d\geqslant 1$ compatible with the $ s$-Riesz DLR equations for $ d-1<s<d$, the  {\it circular Riesz gas,} which is not number rigid. Hence we can use the necessary condition of Theorem \ref{thm:converse-rigidity} to deduce statements about $ \S$, and therefore $ \C =(2\pi )^{d}  \dot  \F\S.$ The fact that Riesz gases are hyperuniform is expected, see an argument in \cite[VI.C]{Lewin}, but there does not seem to be a rigourous proof.

 \begin{corollary}Let $ \P$ a hyperuniform    stationary isotropic point process that is not number rigid in dimension $ d\in \{1,2\}$. Then
 \begin{align*}
\int_{ } | t |^{d}  | \C | (dt) = \infty .
\end{align*}
 \end{corollary}  
  
  \begin{proof} 
 If 
\begin{align*}\int_{ \mathbb{R}} | t |  | \C(dt) | <\infty 
\end{align*}in dimension $ 1,$
$ \S =   \F \C$ has a differentiable density $ \s$ satisfying $ \s(0) = 0$ (hyperuniformity), and
\begin{align*}
 | \s_{}(u)  | \leqslant    \int_{ }|u| | t |  | \C(t) | dt,
\end{align*}
 hence  $ \s_{}^{-1} $ is not integrable in $ 0$, and according to Theorem \ref{thm:main-krigid-iso}, $ \M$ is number rigid, we reach a contradiction.

In dimension $ 2$, if 
\begin{align*}
\int_{\mathbb{R}^{2} }\|t\|^{2} | \C | (dt)<\infty ,
\end{align*} $ \S$ is twice differentiable, and by symmetry $\partial _{1} \s_{}(0) = \partial _{2}\s(0) = 0$, hence
\begin{align*}
\s_{}(u)\leqslant c\|u\|^{2},
\end{align*}
similarly $ \s_{}(u)^{-1}  $ is not integrable in $ 0$, $ \M$ is number rigid, which is again a contradiction.

  \end{proof}

The sine$ _{\beta }$ gases in 1D are proved by  \cite{DHLM,ChhaibiNaj} to be number rigid, and the proof of  \cite{ChhaibiNaj} proves that this rigidity is linear.  
The following statement generally concerns  Coulomb gases which are number rigid and have rotational invariance.
\begin{corollary}
 For a stationary isotropic point process in $ \mathbb{R}^{d}$ which is number rigid and satisfies \eqref{eq:weak-hu}, the spectral density $ s$ satisfies for $ \varepsilon >0$
\begin{align*}
\int_{B(0,\varepsilon ) } \frac{ 1}{\s (u)}du = \infty .
\end{align*}
 \end{corollary}  
 In dimension $ 1$, that seems to be in agreement with predictions and existing results, see  \cite[IV-A-2d]{Lewin}. There do not seem to be such estimates in dimension $ 2$ beyond the Ginibre case $ \P_{Gin}$, corresponding to $ \S(du) = (1-e^{-\|u\|^{2}})du$ and
 $ \beta  = 2$ (Example \ref{ex:GP}).

  \cite{DHLM} prove that the $ sine_{\beta }$ process is not further rigid, and in particular not (linearly) 1-rigid. This is actually expected, but it still rigourously prevents some behaviours.  
  \begin{corollary}Let $ \P$ a  process satisfying \eqref{eq:weak-hu} in dimension $ 1$ that is not 1-rigid. Then
\begin{align*}
\int_{ } | t | ^{3} | \C | (dt) = \infty \textrm{ or }\int_{ }t^{2}\C(t)dt \neq  0.
\end{align*}
\end{corollary}
  
  At first sight, the second moment vanishing  seems unlikely, but this actually occurs for the model $ \P_{\textrm{GAF}}$ (see Example \ref{ex:GP}), and hyperuniformity itself is the result of the vanishing of a moment (of order $ 0$). 
%
%
%
%
Such moment cancellations,  are sometimes called  {\it sum rules} by physicists, and wether they occur and yield higher order rigidity is a fascinating topic of mathematics, see again \cite[VI.C]{Lewin} for references.
%
%

 \subsection{Maximal rigidity of quasicrystals}
 
Quasicrystals are broadly speaking non-periodic atomic measures whose spectrum is purely atomic, their study emerged after experimental discoveries in physics in the 80s and is related to many fields, including crystallography, aperiodic tilings, almost periodicity, see the mathematical monograph  \cite{baake-book} and references therein. Such objects are traditionally assumed to be homogeneous in space, and it is thus natural to consider random constructions that are invariant under translations  (\cite{HartBjo,torquato-quasi}). In the current work, we only need the following wide definition.
\begin{definition}
A WSS quasicystal is a WSS  random measure $ \M$  which spectral measure $ \S$ is purely atomic.
\end{definition}  

There are many possible refinements on this definition, depending on whether the supports of $ \mu $ or $ \F\mu $ are uniformly discrete (i.e.$ \inf_{x\neq y\in  { \rm supp}(\mu )   }\|x-y\|>0)$, or wether $ \mu $ is a counting measure ($ \mu (A)\in \mathbb{N}$ for $ A$ bounded).
See
 \cite{HartBjo} for examples built by a cut-and-project method. They in particular ask in Qestion 1.12 if such processes are number rigid. We bring the following more generaly answer:

 \begin{corollary}
 A WSS quasicrystal   is maximally rigid on any bounded $ A\subset \mathbb{R}^{d}$.
 \end{corollary}

 This result is a simple consequence of Theorem \ref{thm:main-krigid-iso} and the fact that the spectral density $ \s$ is zero by assumption,  hence $ \M$ is $ k$-rigid for any $ k\in \mathbb{N}$ on any compact set by Theorem \ref{thm:main-krigid-iso}.
 We aim to show in the forthcoming paper  \cite{rigid-companion} that  $ \M$ is actually  maximally rigid on much larger sets $ A$, in particular $ A^{c}$ can be an arbitrarily small cone with non-empty interior.

\section{Proofs}  
\label{sec:theory}

For $ \S$ a non-negative measure on $ \hat \e^d=  \mathbb{R}^{d}$ or $ \hat \e^d=   \mathbb  T  ^{d}$ and $ B\subset \hat  \e^{d}$, write $ L^{2}(\S;B)$ as the space of functions $ \psi :B\to  \mathbb C $ satisfying 
\begin{align*}\int_{ B} | \psi  | ^{2}d\S<\infty ,
\end{align*}
simply note $ L^{2}(\S) = L^{2}(\S; \hat \e)$. For a non-negative function $ f: \hat \e\to \mathbb{R}\cup \{ + \infty \}$, write $ L^{2}(f) = L^{2}(f\Leb).$

  We start with the discrete case, which does not require high level Fourier technology, and can serve as a model for the proof of the main theoretical result, Theorem \ref{thm:heart-article}.
  
  \subsection{Discrete processes ($ \hat \e=  \mathbb  T  $)}  
  \label{sec:discrete-prf}

 \newcommand{\co}{c}  
  
   \begin{proof}[Proof of Theorem \ref{thm:discrete} ]

For $ A\subset \mathbb{Z} ^{d}$, recall that $ \E(A)$ is the space of trigonometric polynomials $ \sum_{\m\in A}a_{\m}e^{i \langle  \m, u  \rangle}$ with   $a_{\m}\in \mathbb{C}$.
 For $ \gamma :\mathbb{Z} ^{d}\to  \mathbb C ,$
$ \gamma $-rigidity means that 
for some  $ h_{n}$ compactly supported in $ A^{c}$, a.s. and in $ L^{2}(\mathbf{P}),$
\begin{align*}
\X(\gamma ) = \lim_{n}\X(h_{n}),
\end{align*}
and this is equivalent to 
\begin{align*}
\inf_{h: (\llbracket m \rrbracket^d)^{c}\to \mathbb{C}}  \textrm{Var}\left( \X(\gamma) -\X(h)\right) = \inf_{ h: (\llbracket m \rrbracket^d)^{c}\to \mathbb{C}}\int_{ \mathbb{T}^{d}}\S(du) | \hat \gamma (u)- \hat h(u) | ^{2} = 0,
\end{align*}
where $ \hat \gamma \in \E( \llbracket m \rrbracket^d )$ and $ \hat h\in \E(( \llbracket m \rrbracket^d )^{c})$. The orthogonal space of $ \E(( \llbracket m \rrbracket^d )^{c})$ in $ L^{2}(\S)$   is 
\begin{align*}H_{m}^{\perp} := \{\varphi\in L^{2}(\S):  \int_{  \mathbb  T  ^{d}}      \bar \varphi  {  \hat h}d\S= 0;h:( \llbracket m \rrbracket^d )^{c}\to \mathbb{C}\textrm{ with finite support }\}.
\end{align*}
For such $ \varphi  $, $ \Psi (du) := \varphi (u)\S(du)$ satisfies 
\begin{align*}
\co_{\m}: =(2\pi )^{-d}\int_{ }  \bar  \Psi (du) e^{\imath \langle \m, u\rangle}  =  
\begin{cases} (2\pi )^{-d}
\int_{  \mathbb  T  ^{d}}\bar	\varphi (u)e^{i\langle \m, u\rangle}\S(du)\leqslant (2\pi )^{-d}\| \varphi \|_{L^{2}(\S)}$  for $ | \m | \leqslant m,\\
0$  if $  | \m  | >m
 \end{cases}
\end{align*}
recalling that $ \S(  \mathbb  T  ^{d}) =  \textrm{Var}\left(\X_{0}\right) = 1.$
It means that $  \Psi  $ coincides as inverse Fourier transform 
 with 
\begin{align*}
\sum_{ | \m | \leqslant m}\co_{\m}e^{i\langle \m, u\rangle}.
\end{align*}
In particular  $ \Psi$ has a density $ \psi \in \E( \llbracket m \rrbracket^d )$, and the negligible set supporting the singular part $ \S_{s}$ of $ \S$ is not charged (i.e. $ \varphi \S_{s} = 0$).
We  also have $ \varphi  = \psi \s^{-1}\in L^{2}(\S)$, hence
\begin{align*}
\int_{ }\frac{  | \psi  | ^{2}}{\s_{}}<\infty .
\end{align*}
We proved indeed that in $ L^{2}(\S,  \mathbb C )$,$$  H_{m}^{\perp} \subset  \{\psi \s_{}^{-1}: \psi \in \E( \llbracket m \rrbracket^d)\cap L^{2}(\s_{}^{-1})\} .$$
For the converse, we clearly have for such $ \psi $ vanishing on the negligible singular part's support, $ \varphi : = \psi \s_{}^{-1}\in L^{2}(\S)$, and for $ h\in \E( (\llbracket m \rrbracket^d )^{c})$
\begin{align*}
\int_{ }  \bar \varphi \hat h\s_{} = \int_{ } \bar \psi \hat h =  0,
\end{align*}
hence the inclusion is an equality.
Since LMR is equivalent to $ H_{m}^{\perp} = \{0\}$, this part of the proof is complete.

Mass rigidity (or $ 0$-rigidity) means $ \gamma $-rigidity for $ \gamma  =1_{ \llbracket m \rrbracket^d  }$, i.e. $ \langle \varphi , \hat \gamma  \rangle = 0$ for $ \varphi \in H_{m}^{\perp}$.   Hence it  means, for all $ \psi  = \sum_{\m}\co_{\m}e^{iu\cdot \m}\in \E( \llbracket m \rrbracket^d)\cap L^{2}(\s_{}^{-1})$, 
\begin{align*}
0 = \langle \psi \s_{}^{-1}, \hat \gamma   \rangle_{\S} = \langle \sum_{\m\in \llbracket m \rrbracket^d} \hat \delta _{m}, \psi   \rangle =  \int_{  \mathbb  T  ^{d}}\sum_{\m}e^{-\imath \langle \m, u\rangle}\sum_{\m'}\co_{\m'}e^{i\m' \cdot u} du= \sum_{\m}\co_{\m} =  \psi (0) .
\end{align*}

$ \k$-rigidity means the same with 
\begin{align*}
0 = \langle \sum_{\m}\m^{\k} \hat \delta _{\m},\psi  \rangle  = \sum_{\m}\m^{\k}\co_{\m}= i^{\k}\partial ^{\k}\psi (0).
\end{align*}

Assume now that $ \s_{}^{-1}$ has a pole of order $ \k$. Let $ \varphi \in H_{m}^{\perp}  \setminus \{0\},\psi  = \varphi \s_{}\in \E( \llbracket m \rrbracket^{d}) $.  In particular, $ \int_{ B(0,\varepsilon )} | \psi |  ^{2}\s_{} ^{-1}<\infty $. By Lemma \ref{lm:polyn-equiv-0} below, there is a polynomial $ Q$ equivalent to $ \psi $ in $ 0$ with same derivatives up to order $ \k$. Therefore, $$ \int_{ } | Q | ^{2}\s^{-1}<\infty $$ and since $ 0$ is a $ \k$-pole, $ \partial ^{\k}\psi (0) = \partial ^{\k}Q(0) = 0.$ We indeed proved $ \k$-rigidity.
\end{proof}

\begin{lemma}
\label{lm:polyn-equiv-0}Let $ \psi(u) = \sum_{\m}a_{\m}u^{\m} $ an analytic function on $  \mathbb C ^{d}$, $ J\subset \mathbb{N}^{d} $ finite.
 There exists $ Q$ a  polynomial equivalent to $ \psi  $ in the neighbourhood of $ 0$ with $ \partial ^{\k}Q(0) = \partial ^{\k}\psi (0),\k\in J$.
 
\end{lemma}

\begin{proof}
     Let $ I = \{\m:a_{\m}\neq 0\}\subset \mathbb{N}^{d}$. It is easy to see that there is a finite set $ I_{0}\subset I$ that dominates $ I$ in the sense that for all $ \m '\in I$, there is some $ \m \in I_{0}$ with $ \m '\preceq \m $. 
     Said differently, there does not exist infinite $ I_{0}\subset \mathbb{N}^{d}$ made up  of extremal points, i.e.  such that every $ \m\in I_{0}$ is not $\preceq$-smaller than all others $ \m'\in I_{0}$.
     
Define $$  Q(u) = \sum_{\m\in I_{0}\cup J}a_{\m}u^{\m}$$ so that   $ Q,\psi $ have the same term of order $ \k\in J .$  Also, by uniform convergence of the series, as $ u\to 0,$
     \begin{align*}
\psi (u) = \sum_{\m\in I_{0}}\left[
a_{\m}u^{\m}(1 + o(1)) 
\right]= Q(u)(1 + o(1)).
\end{align*}
\end{proof}

\begin{proof}[Proof of Proposition  \ref{prop:spec-1D}]

Assume the poles of $ \s_{}^{-1}$ in $  [-\pi ,\pi ) $ have finite orders $ k_{i}$,   i.e. for $ \varepsilon >0$ sufficiently small,
\begin{align*}
\int_{ u_{i}- \varepsilon }^{ u_{i} + \varepsilon }\frac{  | u-u_{i} | ^{2k_{i}}}{\s_{}(u)}du<\infty 
\end{align*}  or equivalently 
\begin{align*}
\int_{ u_{i}- \varepsilon }^{u_{i} + \varepsilon }\frac{  | e^{iu}-e^{iu_{i}} | ^{2k_{i}}}{\s_{}(u)}du<\infty 
\end{align*}
with $ \sum_{i}k_{i}\leqslant 2m$. Since $ \s_{}$ is symmetric,  for $ u_{i}$ a pole, $ -u_{i}$ is also a pole (with same order). Let $ u_{1},\dots ,u_{p}$ the poles $ \neq 0$ (repeated according to multiplicity), with $ p\leqslant m$. If $ p = m$, $ 0$ is not a pole otherwise there are more than $ 2m$ poles. If $ p<m$, if $ 0$ is a pole, its order is $2 (m-p)$ or less by the hypothesis. We hence define $ \psi \in \E( \llbracket m \rrbracket)$ by
\begin{align*}
\psi (u) =  | e^{iu}-1 | ^{2(m-p)}\prod_{i = 1}^{p}(e^{iu}-e^{iu_{i}})(e^{-iu}-e^{iu_{i}}).
\end{align*}
$ \psi $ is by construction  of degree $ m$, non null, and in $ L^{2}(\s_{}^{-1})$, hence $ \X$ is not LMR by Theorem \ref{thm:discrete}. 

Conversely, assume  not LMR: by Theorem \ref{thm:discrete} there is $$  \psi (u) = \sum_{ | \m | \leqslant m}\co_{\m}e^{i\langle \m, u\rangle}\not\equiv 0$$ in $ L^{2}(\s_{}^{-1})$, and since $ \s_{}(u) = \s_{}(-u)$,  $ \psi (-u)\in L^{2}(\s_{}^{-1})$. Now, either $ \psi (u) = -\psi (-u)$, or  $ \tilde \psi (u) : =  \psi (u) + \psi (-u)$ is a non-zero even function of $ L^{2}(\s^{-1})$. In any case, we can assume  $  | \psi (u) |  = |  \psi (-u) | $ without loss of generality. 
Let  
\begin{align*}
P(z) = z^{m}\sum_{\m}\co_{\m}z^{\m} = \prod_{i = 1}^{2m}(z-z_{i}) 
\end{align*}
for some $ z_{1},\dots ,z_{2m}\in  \mathbb C $,
which is a polynomial of degree $ 2m$, and $ \psi (u) = P(e^{iu})e^{-imu}$.

 A zero of $ \psi $ is necessarily a root of $ P$  with norm $ 1$,  $ z_{i} = e^{iu_{i}}$, and $ -u_{i}$ is also a zero, hence $ e^{-iu_{i}}$ is also a root of $ P$. Since $ P$ has at most $ 2m$ roots, $ \psi $ has at most $ 2m$ zeros, which means that $ \s_{}$ has at most $ 2m$ poles (all counted with multiplicity).

%

For the converse sense for $ 0$-rigidity, assume that $ \X$ is not LMR, meaning there are less than $ 2m$ poles. If $ 0$ is not a pole, then indeed with the previous notation $ \psi (1) = P(e^{iu\cdot 0})\neq 0$, whence  $ |  \psi (u) |  = |  P(e^{iu}) | $  is in $ \E( \llbracket m \rrbracket)\cap L^{2}(\s_{}^{-1})  \setminus \{0\}$ with 
\begin{align*}\langle \sum_{  | k | \leqslant m}\delta _{k}, \hat \psi  \rangle = \psi (1)\neq 0,
\end{align*}
it gives the equivalence for $ 0$-rigidity.

  \end{proof}
  \subsection{ Schwartz's Paley-Wiener Theorem}

The Fourier transform used here relies on the space of tempered distributions,
denoted by $ \sc'(\mathbb{R}^{d})$ as the dual of the Schwartz space $ \sc(\mathbb{R}^{d})$ of $ \mathcal{C}^{\infty }$; there is no need to enter into technical details in this article,  see for instance  \cite{DuiKolk} for a theoretical exposition, we recall the necessary facts here, beginning by the fact that every $ \mathcal{C}^{\infty }$ function with compact support is a Schwartz function.
We only consider in this paper complex-valued measures, that is a distribution $ \Psi$  on $ \mathbb{R}^{d}$ such that for some non-negative measure $  | \Psi | $ on $ \mathbb{R}^{d}$,
\begin{align*}
 |   \Psi (f) | \leqslant \int_{ } | f |  | \Psi | (dx),f\in \C_{c}^{b}(\mathbb{R}^{d}).
\end{align*}
Such  $\Psi  $   is indeed  {tempered} if  for some $ p>0$
\begin{align}
\label{eq:suff-tempered}
\int_{\mathbb{R}^{d}}(1 + \|u\|)^{-p} | \Psi  |  (du)<\infty ,
\end{align}
it is in particular locally finite.
We actually only consider  measures of the form $ \Psi = \psi \S$ where $ \psi :\mathbb{R}^{d}\to  \mathbb C $ and $ \S$ is a non-negative measure.

For $ \Psi$ tempered, one can  define the Fourier transform $ \F\Psi\in \sc'(\mathbb{R}^{d})$ through 
\begin{align*}
\F\Psi(\varphi )   = \int_{ } \hat \varphi d\Psi
\end{align*}
for $ \varphi \in \C_{c}^{\infty }(\mathbb{R}^{d})$ and $ \hat \varphi(u) = \int_{ \mathbb{R}^{d}}e^{-i \langle u,t \rangle}dt $ is the usual Fourier transform on $ L^{2}(\mathbb{R}^{d})$. 
Recall the Fourier inversion:
\begin{theorem}[\cite{DuiKolk}, Th.14.18]
 For $ \Psi\in \sc'(\mathbb{R}^{d})$,
\begin{align*}
\F( \F  \dot \Psi )= (2\pi )^{d}\Psi
\end{align*} 
where $  \dot \Psi  = \Psi(-\,. )$ in  the distributional sense. 
 Call  {\it spectrum} of $ \Psi $, denoted by $ \sp(\Psi )$, the support of $ \F\Psi $ as a distribution, i.e. the largest closed set $ A$ such that $ \langle \F\Psi,\varphi  \rangle = 0$ for $ \varphi $ supported by $ A^{c}.$
 \end{theorem}

For a compact $ A,$ define for $   \zeta \in \mathbb{R}^{d} $ 
\begin{align*}s_{A}( \zeta ) = \sup_{x\in A} \langle  \zeta ,x \rangle.
\end{align*}

In the case $ A = B(0,R)$, we have  $ s_{A}(\zeta ) = R\|\zeta \|$. One can also notice that $ s_{A} = s_{  \textrm{conv}(A)}$ characterises the closed convex hull $ \conv(A)$ of $ A.$ 

Call analytic function on $ \mathbb{C}^{d}$ any  function 
\begin{align*}
\psi (z) = \sum_{\m}a_{\m}z_{1}^{\m_{1}}\dots z_{d}^{\m_{d}},z\in  \mathbb C ^{d}
\end{align*}
where $ a_{\m}\in   \mathbb C $ and  the series is absolutely convergent on $  \mathbb C ^{d}.$  Recall that by Hartog's theorem (see  \cite[Sec.0.2]{Krantz}), this is equivalent to the analycity of $ \psi $ in each variable $ z_{i}$ separately. The scalar product on $ \mathbb{C}^{d}$ is    $ \langle x,z \rangle = \sum_{i}\bar x_{i}z_{i}$. For $ z = (z_{1},\dots ,z_{d})\in  \mathbb C ^{d},$ write 
$ \|z\|_{  \mathbb C }^{2} = \langle  \bar z,z \rangle$, not to be mistaken with the entire function $ \|u\|^{2} = \sum_{i}u_{i}^{2}$, more useful  on $ \mathbb{R}^{d}$. \begin{theorem}
\label{thm:SPW}
[Schwartz's Paley-Wiener Theorem] Let $ A\subset \mathbb{R}^{d}$ bounded.
Let $ \Psi$ a  complex-valued    measure on $ \mathbb{R}^{d}$. Then $ \Psi$ is tempered with $ \sp(\Psi )\subset \conv(A)$  if and only if  $ \Psi$  has a tempered density $ \psi $ wrt $ \Leb$ on $ \mathbb{R}^{d}$
 that   can be extended as an analytic function on $   \mathbb      C^{d}$ such that for some finite $ C,$ 
\begin{align*}
 | \psi (z) | \leqslant C
 \exp(s_{A}(  \mathscr  Iz)),z\in  \mathbb C^{d}.
\end{align*}

\end{theorem}  
 Such functions are denoted by $ \E(A)$, which is therefore a subclass of the class of entire functions of $ \mathbb{R}^{d}$, and also a subclass of the class of (restrictions to $ \mathbb{R}^{d}$ of) analytic functions with exponential type, and the extension to $  \mathbb C ^{d}$ is still denoted $ \psi $.

   An important observation is that for  $\psi \in \E(A) $, for a polynomial $ Q$, $ Q\psi \in \E(A)$, hence it still has spectrum in $ A$, and if $ \psi Q^{-1}$ is analytic, it is also in $ \E(A).$
 
 A more general version of the previous theorem for all $ \sc'(\mathbb{R}^{d})$ is proved at  \cite[Th. 17.1, Theorems 17.3]{DuiKolk}.

The construction of some functions of $ \E(A)$ rely on the following technique, involving $ J_{d}$ the Fourier transform of the unit sphere indicator:
\begin{align}
\label{def:Jd}
J_{d} (u):= \frac{ B_{d/2}( \| u \| )}{ \| u \| ^{d/2}},u\in \mathbb{R}^{d},
\end{align}
where $ B_{d/2}$ is the Bessel function of the first kind of order $ d/2$.

\begin{lemma}
\label{lm:psi-to-tilde}
Let $ \psi \in \E(A)$ for some compact $ A$. Then for every non-zero polynomial $ P$ and $ \eta >0,p\geqslant 0$, the function 
\begin{align*}
\tilde \psi (u) = P(u)\psi (u)J_{d}(\eta u /M)^{M}
\end{align*} 
with $ M = \frac{ 2}{d + 1}(\deg(P) + p)$
satisfies   $  | \tilde \psi(u)  | \leqslant c | \psi(u)  |\|u\|^{-p} $ for some finite $ c$ and $ u\in \mathbb{R}^{d}$, and $ \psi \in \E(A^{ + \eta })$  where $ A^{ + \eta }: = \cup _{t\in A}B(t,\eta ).$
\end{lemma}

\begin{proof}By Theorem \ref{thm:SPW}, $ J_{d}$ is an analytic function with spectrum in $ B(0,1)$, hence $ J_{d}(\eta \cdot /M)^{M}\in \E(B(0,\eta )).$ Hence $ \tilde \psi $ is analytic as well, and using again Theorem \ref{thm:SPW}, it proves that
\begin{align*}
 | \tilde \psi (z) | \leqslant c \exp(s_{A}(z))\exp(\eta \|z\|).\end{align*}
Then, for some $ x_{0}\in A,$ $$  s_{A}(z) + \eta \|z\| =  \langle x_{0},z \rangle + \eta \|z\| \leqslant \sup_{x\in A,t\in B(0,\eta )} \langle x + t,z \rangle  =  s_{A^{ + \eta }}(z)$$ hence
 $ \psi \in \E(A^{ + \eta }).$
We have the classical equivalent as $ u\to \infty $, for some $ c,c_{d}>0,$
\begin{align}
\label{eq:bd-bessel}
J_{d}(u) = c \| u \| ^{-(d + 1)/2}\cos( \| u \| -c_{d})(1 +O(\|u\|^{-1}) ),
\end{align}
which gives $  | \tilde \psi(u)  | \leqslant c | \psi(u)  |\|u\|^{-p}.$

\end{proof}

  \subsection{Characterisation of $ \gamma $-rigidity}  
  
     Let us translate the general rigidity problem in the Fourier space. Let $ A\subset \mathbb{R}^{d}$. For $ \gamma \in \C_{c}^{b}(\mathbb{R}^{d})$, $ \gamma $-rigidity on $ A$ means  
\begin{align*}
\inf_{h \in \mathcal{C}_{c}^{\infty }(A^{c})}  \textrm{Var}\left( \int_{A}\gamma (t)\M(dt)-\int_{A^{c}}h(t)\M(dt)\right) = &0.
\end{align*}

It implies that for some sequence $ h_{n}, I_{\M}(\gamma )-I_{\M}(h_{n})\to 0$ in $ L^{2}(  \mathbf  P  )$, hence for some subsequence we have a.s. $ I(\gamma ) = \lim_{n'\to \infty }I(h_{n'})$, meaning $ I(\gamma )\in \sigma_{lin} (\M_{A^{c}})$.  By \eqref{eq:phase-formula-SF}, $ \gamma $-rigidity is equivalent to 
\begin{align*}\inf_{h \in \mathcal{C}_{c}^{\infty }(A^{c})} \int_{}\S(du) | { \hat \gamma (u)}- \hat h(u)| ^{2} = &0.
\end{align*}

The heart of this article lies in the current characterisation of $ \gamma $-rigidity. Recall that the spectral density  $ \s_{}\geqslant 0$ is the density of the continuous part of $ \S.$

 \begin{theorem}
 \label{thm:heart-article}
Let $ A$ a convex body.
 $ \M$ is linearly $ \gamma $-rigid on 
 $ A$  if and only if  for all $ \psi\in \E(A)\cap  L^{2}(\s_{}^{-1})$, 
\begin{align*}\int_{A}  \hat  \gamma  \bar \psi   = 0.
\end{align*}
   
  \end{theorem}

  Before the proof, an instructive immediate corollary:
\begin{remark}
\label{prop:rigid-monotone}
For $ 0\leqslant \s_{}\leqslant \s_{}'$, $ \gamma $-rigidity for $ \s_{}'$ on  $ A$ implies $ \gamma $-rigidity for $ \s_{}$ on $ A.$

\end{remark}

We also need a lemma:
\begin{lemma}
\label{lm:decrease-S-temp}
For any spectral measure $ \S$ of a $ L^{2}$ wide-sense  stationary random measure $ \M$, 
\begin{align*}
\int_{\mathbb{R}^{d}}(1 + \|u\|)^{-(d + 1)}\S(du)<\infty .
\end{align*}
In particular, $ \S$ is a tempered measure.
\end{lemma}

\begin{proof}Let $ J_{d}$ the Fourier transform of the indicator of the unit ball (see  \eqref{def:Jd}), satisfying in particular \eqref{eq:bd-bessel}. Let $ z_{0} = (\pi /2,\dots ,\pi /2)$ and $$  \psi (u) =J_{d}( u )^{2} + J_{d}(u + z_{0})^{2}.  $$ The function $ \cos(\|u-c_{d}\|)^{2} + \cos(\|u-c_{d} + z_{0}\|)^{2}$ is larger than some $ \kappa >0$, hence for some $ R>0$,  $ \psi $ satisfies for $ \|u\|\geqslant R$, for some $ 0<c_{-}\leqslant c_{ + }<\infty ,$
\begin{align*}
\frac{ c_{-}}{ \| u \| ^{d + 1}}\leqslant \psi (u)\leqslant 
\frac{ c_{ + }}{ \| u \| ^{d + 1}}.
\end{align*}
Also  $ {{ \widehat{ \psi }}}  $ is bounded   by $ 2$ and $ \sp(\psi )\subset \sp(J_{d}^{2})\subset  B(0,2)$. Therefore, by  \eqref{eq:phase-formula-SF}
\begin{align*}
\int_{}\S(du)(1 + \|u\|)^{-(d + 1)}\leqslant & \S(B(0,R))  + c_{-}^{-1}\int_{B(0,R)^{c}}  \psi (u)\S(du) \\
\leqslant &\S(B(0,R)) + c_{-}^{-1}(2\pi )^{d} \left(
\textrm{Var}\left(I_{\M}( { 1_{B(0,1)}})\right) +  \textrm{Var}\left(I_{\M}( e^{\imath \langle z_{0},\cdot  \rangle}1_{B(0,1)})\right)
\right)<\infty .
\end{align*}

\end{proof}

 \begin{proof}[Proof of Theorem  \ref{thm:heart-article} ]

Denote by $  \bar H^{\S}$ the closure of some subspace $ H$ of $ L^{2}(\S)$, and let $$  H_{A} = \{ \hat h: h\in \mathcal{C}_{c}^{\infty }(A^{c})\}.$$
This is indeed a subspace of $ L^{2}(\S)$ because for $ h\in \mathcal{C}_{c}^{\infty }(A^{c})$, $ \hat h \in  \mathscr  S(\mathbb{R}^{d})$ has rapid decay, and
\begin{align*}
\int_{ \mathbb{R}^{d}}  | \hat h | ^{2}\S\leqslant c\int_{ \mathbb{R}^{d}}(1 + \|u\|)^{-(d + 1)}\S(du)<\infty 
\end{align*}
by Lemma  \ref{lm:decrease-S-temp}.

Hence $ \M$ is $ \gamma $-rigid  if and only if  $ \hat \gamma \in  \bar H_{A}^{\S}$,  if and only if  $ \int_{ } { \hat \gamma }  \bar \varphi \S = 0$ for $ \varphi \in H_{A}^{\perp}$ where $$  H_{A}^{\perp} = \{\varphi \in L^{2}(\S ):\int_{ \mathbb{R}^{d}}  \bar \varphi  { \hat h}  d\S= 0;h\in \mathcal{C}_{c}^{\infty }(A^{c})\}.$$ The proof of the theorem is concluded  by the following lemma:

\begin{lemma}
\label{lm:main-Hperp}Let $ A\subset \mathbb{R}^{d}$ bounded measurable. 
\begin{align*}H_{A}^{\perp} =\{\psi \s_{}^{-1} :\psi \in \E(\conv(A))\cap L^{2}(\s_{}^{-1}) \}.
\end{align*}  
\end{lemma}

\begin{proof}
[Proof of Lemma \ref{lm:main-Hperp}]

Let $ \varphi \in  H_{A}^{\perp}.$
Let $ \Psi  = \varphi \S.$  The measure $    \Psi    $ is  tempered  because by Cauchy-Schwarz inequality in $ L^{2}(\S),$
\begin{align*}\int_{\mathbb{R}^{d}} | \varphi (u)|  \S (du)( 1 + \| u \|) ^{-(d + 1)/2}\leqslant {\|\varphi \|_{L^{2}(\S)}}\sqrt{\int_{} (1 +  \| u \| )^{-(d + 1)}\S(du) }
\end{align*}
and the latter is finite with Lemma  \ref{lm:decrease-S-temp}.
  Let $  \F \Psi $ its Fourier transform in the sense of $ \sc'(\mathbb{R}^{d})$.
Let $ h\in  \mathcal{C}_{c}^{\infty }(A^{c})\subset \sc(\mathbb{R}^{d})$, then $ \hat h\in  \sc(\mathbb{R}^{d})\cap L^{2}(\S)$. 
By definition of $ H^{\perp}_{A}$, \begin{align*}0 = \int_{ }  \hat h\bar  \varphi d\S=\int_{ }  \hat hd\bar\Psi= \overline{\F\Psi(h)}.
\end{align*}
Hence $ \sp(\Psi )\subset A$.
 By  Theorem \ref{thm:SPW}, $ \Psi$ has a tempered density $ \psi :$ $$ \Psi  = \psi \Leb  = \varphi \s_{}\Leb$$ and $ \psi $ is a function of $ \E(  \textrm{conv}(A)).$
 Also, since $ \varphi \in L^{2}(\S)$, 
\begin{align*}
\int_{}\frac{ \psi^{2} }{\s_{} } = \int_{}\varphi ^{2}\s_{} =  \int_{}\varphi ^{2}\S<\infty ,
\end{align*}indeed $ \psi \in L^{2}(\s_{}^{-1})$ (with $ 1/0 = \infty $).

For the converse, let $ \psi \in L^{2}(\s_{}^{-1}) \cap \E(\conv(A))$, in particular $ \sp(\psi )\subset  \textrm{conv}(A)$, and let $ \varphi  = \psi \s_{}^{-1}\in L^{2}(\S)$. For $ h\in  \mathcal{C}_{c}^{\infty }(A^{c})\subset \mathscr  S(\mathbb{R}^{d})$,  
\begin{align*}0 = \F \psi(h)     = \langle \psi , \hat h \rangle= \int_{ }  \bar \varphi \hat h d\S,
\end{align*}
indeed $\varphi \in H_{A}^{\perp}$.
\end{proof}
\end{proof}

%
%
%


%
%

\subsection{Proof of Theorem  \ref{thm:main-krigid-iso}} 
\label{sec:prf-main-krigid}

We assume here that $ \s$ has a pole of order $ \k$ in $ 0.$ Let us prove $t^{\k} $-rigidity. It is enough to prove $ \gamma $-rigidity for some $ \gamma \in \sc(\mathbb{R}^{d})$ such that $ \gamma (t) = t^{\k}$ for $ t\in A$. 
We actually prove that it is $ \gamma $-rigid on $  \textrm{conv}(A)$, which in turns implies $ \k$-rigidity on $ A.$
Hence without loss of generality we assume $ A$ is convex, mostly for notational simplification.
We use  Theorem  \ref{thm:heart-article}: we must prove that $ \hat  \gamma $ is orthogonal to all $   \psi \in L^{2}(\s^{-1})\cap \E(A)$. For technical reasons we actually assume that $ \gamma (t) = t^{\k}$ on $ A^{1}:=\{x:d(x,A)\leqslant 1\}$. We use the following lemma.\begin{lemma}
\label{lm:ipp-deriv}
Let $ \gamma \in \C_{c}^{\infty }(\mathbb{R}^{d})$ coinciding with $ t^{\k}$ on $ A^{1}$ and $ \psi \in L^{2}(\s^{-1})\cap  \E(A)$. Then  
\begin{align*}
\int_{ } \psi \hat \gamma  = (-\imath)^{  | \k |}\partial ^{\k}\psi (0).
\end{align*}

\end{lemma}
Let us conclude the proof before proving the lemma: 
Since $ \psi \in L^{2}(\s_{}^{-1})$ and is analytic, we have by Lemma \ref{lm:polyn-equiv-0} for some polynomial $ Q$ with $ \partial ^{\k}Q(0) = \partial ^{\k}\psi (0)$ and for $ \varepsilon $ sufficiently small
\begin{align*}
\int_{ B(0,\varepsilon )}\frac{  | Q(u) | ^{2}}{\s_{}(u)} du< \infty ,
\end{align*}
by definition of $ \k$-incompatibility it forces $ \partial ^{\k}Q(0) = 0= \partial ^{\k}\psi (0)  = \int_{ }\psi \hat \gamma $, for all $ \psi \in \E(A)\cap L^{2}(\s^{-1})$. We proved that $ \M$ is  $ \gamma $-rigid, hence $ \k$-rigid on $ A$.

\begin{proof}[Proof of Lemma  \ref{lm:ipp-deriv}]
The proof would be elementary if $ \psi $ had fast decay (see \eqref{eq:xxx}), but the generality of what can be $ \psi $ requires fine approximation arguments based on the hypotheses.
We approximate $ \psi $ by $  \psi _{n}:= \psi f_{n}$    where $$  f_{n}(u)=\exp(-\|u/n\|^{2k + 2}).$$
The first point is that $ \partial ^{\k}\psi _{n}(0)=\partial ^{\k}\psi (0)$ as the first derivatives of $ f_{n}$ vanish in $ 0$, except for $ f_{n}(0) = 1.$
Let us show that $ f_{1}$ is a Schwarz function: all its derivatives clearly have fast decay, and for the Fourier transform,
for $ N>0,\m  \in \mathbb{N}^{d}$, 
\begin{align*}
\sup_{t}  \|t\|^{2N}  |\partial  ^{\m} \hat f_{1}  (t)| 
=& \sup_{t}  (\|t\|^{2})^{N}  | \widehat{ u^{\m} f_{1}}(t)  | 
\\
\leqslant & \sup_{t}\left|
\int_{}u^{\m}e^{-\|u\|^{2(k + 1)}}(  \|t\|^{2})^{N} e^{-i \langle t,u \rangle}du
\right| \\
= &\sup_{t}\left|\int_{}
u^{  \m  }e^{-\|u\|^{2(k + 1)}}(-1)^{N}(  \sum_{i = 1}^{d}\partial _{u_{i}}^{2})^{N} e^{-i \langle t,u \rangle}du
\right|\\
 \leqslant  & \int_{} |  ( \sum_{i}\partial ^{2}_{u_{i}})^{N}(u^{\m}e^{-\|u\|^{2(k + 1)}}) | du
\end{align*}with multiple integration by parts,
whence
\begin{align*}
C_{N}^{\m}: = \sup_{t}(1 + \|t\|)^{2N} | \partial ^{\m} \hat f_{1}(t) | <\infty .
\end{align*}
 Since $ \psi \in \E(A)$, Theorem \ref{thm:SPW} yields
\begin{align*}
 | \psi (u) | \leqslant C\exp(C\|u\|),u\in \mathbb{R}^{d},
\end{align*} hence $ \psi _{n}, \hat \psi _{n}$ both are $ \mathcal{C}^{\infty }$ with superpolynomial decay (recall that $ \hat \psi $ has bounded support).
This yields with standard arguments of classical Fourier transform
\begin{align}
\notag(-\imath)^{ | \k | }\partial ^{\k}\psi (0)=(-\imath)^{ | \k | } \partial ^{\k}\psi _{n}(0)=& (2\pi )^{d}\int_{} \hat \psi _{n}(t)t^{\k}dt\\
\label{eq:xxx}=&(2\pi )^{d}\int_{} \hat \psi _{n}(t)\gamma (t)+(2\pi )^{d}\int_{(A^{1})^{c}} \hat \psi_{n}(t) (t^{\k}-\gamma(t))dt 
\end{align}
using that  $ \gamma (t) = t^{\k}$ on $ A^{1}.$ 
Let us show that the first term converges to $ \int_{}\psi \hat \gamma .$
Recall that by Lemma \ref{lm:decrease-S-temp}, $$  \int_{}(1+\|u\|)^{-(d+1)}\s(u)du<\infty $$
and that $ (1+\|u\|)^{K} \hat \gamma (u)$ is bounded for any $ K>0$.  We also exploit $ \psi \in L^{2}(s^{-1})$ (with $ \psi (u)=0$ if $ \s(u)=0$), and Cauchy-Schwarz inequality: for some $ c<\infty ,$
\begin{align*}
&\left|(2\pi )^{d}
\int_{} \hat \psi _{n} \gamma -\int_{}
\psi \hat \gamma \right| =\left|
\int_{}\psi  \hat \gamma (1-f_{n})
\right|\\
\leqslant& \int_{} | {\psi } |  |  \hat \gamma (1-f_{n}) \s | {\s^{-1}}\\
\leqslant & \|\psi \|_{L^{2}(\s^{-1})}\sqrt{\int_{} | \hat \gamma  | ^{2} | 1- f_{n}| ^{2}  \s}\\
\leqslant & \|\psi \|_{L^{2}(\s^{-1})}\sup_{u}| \hat \gamma (u)(1+\|u\|)^{\frac{d+1}{2}} | \sqrt{\int_{} | 1-f_{n} | ^{2} 
(1+\|u\|)^{-(d+1)}
\s(u)du}\\
\leqslant &c\Big(\sup_{u\in B(0,\sqrt{n})} | 1-f_{n}(u) | ^{2}\int_{B(0,\sqrt{n})}(1+\|u\|)^{-(d+1)}\s(u)du
+\int_{B(0,\sqrt{n})^{c}}(1+\|u\|)^{-(d+1)}\s(u)du\Big)^{1/2}
\end{align*}
and all terms indeed go to $ 0.$

It remains to show that the second term of  \eqref{eq:xxx} goes to $ 0.$
 We use the fact that since $ \F\Psi $ is a distribution with compact support in $ A$, there is $ C<\infty ,m\in \mathbb{N}$ such that for $ N\in \mathbb{N}, t\in (A^{1})^{c}$
\begin{align*}
 | \hat \psi _{n}(t) | = &| \F\psi (\tau _{t} \hat f_{n}) | \leqslant \sup_{ | \m  | \leqslant m}C\|\partial ^{\m } \tau _{t}\hat f_{n}\|_{A} \leqslant C \sup_{ | \m  | \leqslant m}\sup_{y\in A}n^{d +  | \m | }C_{N}^{\m}(1 + n\|t-y\|)^{-2N},\\
 \leqslant & C\sup_{ | \m | \leqslant m }n^{d+ m-2N}C_{N}^{m}(1+nd(t,A))^{-2N}.
\end{align*}
Choosing $ N>\max(d+m, | \k | +2d)/2$ yields that the second member of \eqref{eq:xxx} goes to $ 0$, thereby concluding the proof.

\end{proof}
%

\subsubsection{  Proof of Proposition  \ref{prop:k-incomp-iso}}

 Let $ Q \in L^{2}(\s_{}^{-1},B(0,\varepsilon ))$ a polynomial, $ q$ the degree of lowest order terms, meaning $ Q = Q_{h} + R$ with $ Q_{h}$ a non-null homogeneous polynomial, i.e. veryfying for some $ q\in \mathbb{N},$ $ u\neq 0,$
\begin{align*}
Q_{h}(u) = \|u\|^{q}Q_{h}(\theta ),\theta  = \frac{ u}{\|u\|}\in  \mathbb  S  ^{d-1},
\end{align*} and $ R(u) = o(\|u\|^{q})$ as $ u\to 0$. We have 
\begin{align*}
\int_{B(0,\varepsilon ) }\s_{}^{-1}(u) | Q | ^{2}(u)du& \geqslant \int_{B(0,\varepsilon ) } \tilde \s(\|u\|)^{-1}\|u\|^{2q}( | Q_{h}(\theta )  | + o(1))^{2}du\\
&\geqslant \sigma _{d}\int_{ 0}^{\varepsilon } \tilde \s(\rho )^{-1}\rho ^{2q}\left(
\int_{  \mathbb  S  ^{d-1}} | Q_{h}(\theta ) | ^{2}d\theta  + o(1)
\right)\rho ^{d-1}d\rho .
\end{align*}
Therefore, if $ q\leqslant k$, indeed $ \s_{}^{-1} | Q | ^{2}$ is not integrable around $ 0$, which means $ Q$ cannot have terms of degree $ \leqslant k$, and $ \s$ is $ \k$-rigid for $  | \k |  \leqslant k.$

   \subsection{Necessity lemma}

   The necessity proofs rely on the following technique.

\begin{lemma}
\label{lm:no-Q-rigid}
Let $ Q$ a  $ \k$-polynomial  (for instance $ Q(u) = u^{\k}$), $ A\subset \mathbb{R}^{d}$ bounded measurable.
Assume for some $ \varepsilon >0$ there is $ \psi \in L^{2}(\s^{-1},\mathbb{R}^{d}  \setminus B(0,\varepsilon ))\cap \E(A)$ and $ \psi (0)\neq 0$, and that $ 0$ is not a pole of order $\k$ for $ \s^{-1}$. Then $ \M$ is not $ Q$-rigid on  $ A^{ +\eta  }$ for $ \eta >0$.\end{lemma}

\begin{proof}

Since $ 0$ is not a pole of order $\k$, it means that there exists a polynomial $ P(u) = \sum_{\k'}b_{\k'}u^{\k'}$ with  $ b_{\k} \neq 0$ and for some $ \varepsilon >0$\begin{align*}
\int_{ B(0,\varepsilon )} | P | ^{2}\s_{}^{-1}<\infty .
\end{align*}
If there is also a term $ b_{\k'} \neq 0$ for some $ \k'\prec\k,$ (i.e.$\k'\preceq\k,\k'\neq \k$), we have
\begin{align*}\tilde P = u^{\k-\k'}P\in L^{2}(\s^{-1},B(0,\varepsilon  )),
\end{align*}
and $  | \tilde P | ^{2}\s^{-1}$ is still integrable on $ B(0,\varepsilon )$,
we can assume without loss of generality that there is no term strictly smaller than $\k$, i.e. 
\begin{align}
\label{eq:tildeP}
P(u)= \sum_{\k'\not\prec\k}b_{\k'}u^{\k'}.
\end{align} Let then 
\begin{align*}
  \tilde  \psi (u)  =   \psi(u) J_{d}(u\eta/\deg(P))^{\deg(P)},
\end{align*}
satisfying $  \tilde\psi (0)\neq 0$. By Lemma  \ref{lm:psi-to-tilde}, $ P  \tilde  \psi\in  \E(B(0,A^{ + \eta } ))$ and $  | P \tilde \psi  | \leqslant c | \psi  | $, hence also $P \tilde \psi \in  L^{2}(\s^{-1},B(0,\varepsilon )^{c})  $.
We also have $ P  \tilde  \psi \in L^{2}(\s^{-1},B(0,\varepsilon ))$ because $  P\tilde  \psi(u)\sim P(u) $ as $ u\to 0$, hence $ P  \tilde  \psi \in L^{2}(\s^{-1})\cap \E(A^{ + \eta }).$

Since $ Q$ is a $ \k$-polynomial, $ Q(u) = \sum_{\m\preceq \k_{0}}a_{\m}u^{\m}$ for some $ \k_{0}\succeq \k$ with $ a_{\k_{0}}\neq 0$. Since $ 0$ is also a pole of order $ \k_{0}$ (because $ \k_{0}\succeq \k$), the whole proof can be done with $ \k_{0}$ instead of $ \k$, and we assume without loss of generality $ \k = \k_{0}$.

Let  $ \gamma \in \sc(\mathbb{R}^{d})$ coinciding with $ Q$ on $ A^{ + \eta }$. By Theorem \ref{thm:heart-article} the proof is complete if we prove that 
$ P \tilde \psi $ (an element of $  L^{2}(\s^{-1})\cap \E(A^{ + \eta }) $) is not orthogonal to $ \hat \gamma $ (and $ \gamma $-rigidity on $ A^{ + \eta }$ is equivalent to $ Q$-rigidity on $ A^{ + \eta }$).

Using lemma  \ref{lm:ipp-deriv} for each $ \m \preceq \k_{0}$,    
\begin{align*}\int_{ } \hat \gamma  P \tilde \psi =\sum_{\m}a_{\m} \int_{A}t^{\m} \widehat{P \tilde \psi }(t)dt=
\sum_{\m}(-\imath)^{\m}a_{\m}\partial ^{\m}(P \tilde \psi )(0) .
\end{align*}

Denote by $  \bar a_{\m}=(-\imath)^{\m}a_{\m}.$
We have with  \eqref{eq:tildeP} \begin{align*}
\sum_{\m}\bar a_{\m}\partial ^{\m}(P \tilde \psi )(0)   =& \sum_{\m\preceq \k}\sum_{\k'\not\prec\k}b_{\k'}\bar a_{\m}\partial ^{\m}(u^{\k'} \tilde \psi  )(0)\\
 =& \bar a_{\k}b_{\k}\partial ^{\k}(u^{\k} \tilde \psi  )(0) + \sum_{\m\preceq \k,\k'\not\prec \k,\k'\neq \k,\m\neq \k}b_{\k'}\bar a_{\m}\partial ^{\m}(u^{\k'} \tilde \psi  )(0).
\end{align*}

For $ \k'\not\prec\k,\m\preceq \k$, we have $ \k'\not\preceq \m$, hence, recalling $ a_{\k}b_{\k}\neq 0,$
the conclusion comes from the fact that $ \tilde \psi (0)\neq 0$ and the following identity:
We have for $ \k'\not\prec  \m$
\begin{align*}
\partial ^{\m}(u^{\k'} \tilde \psi   )(0) =\delta _{\m = \k'}  \m! \tilde \psi  (0)
\end{align*}
with $ \m! = m_{1}\dots m_{d}.$
Let us finally prove this identity.

The assumption yields that we either have $ \m=\k' $
 , or
$ \m_{i}<\k'_{i}$ for some $ i$. Then the Leibniz formula for functions of one argument yields,  with  the notations $ \k'_{ \hat i} = (\k'_{1},\dots ,\underbrace{{0}}_{i},\dots ,\k'_{d}),$ $u_{ \hat i}  =  (u_{1},\dots ,\underbrace{0}_{i},\dots ,u_{d}),u=(u_{i},u_{ \hat i}),$
\begin{align*}
\partial ^{\m}(u^{\k'}\tilde \psi  (u)) =& (\prod_{j\neq i}\partial _{j}^{\m_{j}}) \partial _{i}^{\m_{i}}(u_{i}^{\k'_{i}}u_{ \hat i}^{\k'_{ \hat i}}\tilde \psi  (u_{i},u_{ \hat i}))\\
 = &(\prod_{j\neq i}\partial _{j}^{\m_{j}}) (\sum_{a>0}\lambda _{a}u_{i}^{a}\tilde \psi  _{a}(u) + 1_{\k'_{i} = \m_{i}}u_{ \hat i}^{\k'_{ \hat i}}\m_{i}!\tilde \psi  (u_{i}, u^{ \hat i}))\textrm{  for some $ \lambda _{a}\in \mathbb{R},\textrm{ functions  }\tilde \psi  _{a}$,}\\
  = &\sum_{a>0}\lambda _{a}u_{i}^{a}(\prod_{j\neq i}\partial _{j}^{\m_{j}}) (\tilde \psi  _{a}(u)) + 1_{\k'_{i} = \m_{i}}\m_{i}!(\prod_{j\neq i}\partial _{j}^{\m_{j}} ) (u_{ \hat i}^{\k'_{ \hat i}}\tilde \psi  (u))\\
\partial ^{\m}(u^{\k'}\tilde \psi  (u)) | _{u = 0}     & 
   \begin{cases} 
 =    0$  if $\m_{i} < \k'_{i}\\
   = \m! \tilde \psi  (0)$  if $ \m = \k'$  (with an induction on $ j$)$.
    \end{cases}
\end{align*}

%
%
\end{proof}
%

 \subsubsection{Proof of Proposition  \ref{prop:not-lmr-finite-z}}
 \label{sec:prf-converse}
  Let $ u_{1},\dots ,u_{m}\in \mathbb{R}^{d}$ the poles of $ \s^{-1}$, let $\varepsilon >0, q$ such that  $ P(u): = \prod_{i}\|u-u_{i}\|^{2q}$ is in $ L^{2}(\s^{-1},B(u_{i},\varepsilon ))$ for all $ i$. Introduce 
\begin{align*}
\psi (u) = P(u)J_{d}(u/\eta M)^{M}
\end{align*}
with $ M = \frac{ d + 1}{2}(\deg(P) + p + d)$. Lemma  \ref{lm:psi-to-tilde} (with $ 1\in \E(\{0\})$) yields $ \psi \in \E(B(0,\eta )) $ and $  | \psi (u) | \leqslant c(1 + \|u\|)^{-p-d}$ ($ \psi $ is smooth hence bounded around $ 0$).
We have
\begin{align*}
\int_{\mathbb{R}^{d}}\frac{   |  \psi (u) |   ^{2}}{\s_{}(u)}du \leqslant &c\int_{\mathbb{R}^{d}  \setminus \cup _{i}B(u_{i},\varepsilon )} |  \psi(u) |  ^{2}(1 + \|u\|)^{p} + \sum_{i}c_{i,\varepsilon }\int_{B(u_{i},\varepsilon )}\frac{ \|u-u_{i}\|^{4q}}{\s_{}(u)}du\\
\leqslant &c \int_{}( 1 + \|u\|)^{-2d-2p}(1 + \|u\|)^{p}du + C\\
<&\infty  ,
\end{align*}
hence $ \psi \in L^{2}(\s^{-1})\cap \E(B(0,\eta ))  \setminus \{0\}.$
Let $ \gamma  = \widehat{    \psi} $, bounded and supported by $ B(0,\eta  )$. $ \M$ is not $ \gamma $-rigid on $ B(0,\eta )$ by Theorem \ref{thm:heart-article} because with Parseval's identity
\begin{align*}
\int_{ }  \bar  \gamma \widehat{  \psi}  = \int_{ } | \psi  | ^{2}\neq 0.
\end{align*}
In particular, $ \M$ is not linearly maximally rigid.

The sufficiency part about $ \k$-rigidity is Theorem \ref{thm:main-krigid-iso}.
For the necessity,  define instead 
\begin{align*}
\tilde \psi (u) = \prod_{i  : u_{i}\neq 0}\|u-u_{i}\|^{2q}J_{d}(u/\eta M)^{M}.
\end{align*}
A similar reasoning shows that $ \tilde \psi $ is in $ L^{2}(\s^{-1},\mathbb{R}^{d}  \setminus B(0,\varepsilon ))\cap \E(B(0,\eta ))$ for some $ \varepsilon >0$ with $ \tilde \psi (0)\neq 0.$  We can then conclude with  Lemma  \ref{lm:no-Q-rigid}.

\subsubsection{Proof of Proposition  \ref{ex:spectral-counterexample}}
\label{sec:prf-counterex}

 The proof consists in showing that $ \M$ is $ \gamma $-rigid for any $ \gamma \in \sc(\mathbb{R}^{d})$ taking value $ 1$ on $ A$, hence $ 0$-rigid on $ A$. For that we must prove by Theorem \ref{thm:heart-article} that $ \int_{ } \hat \gamma \psi  = 0$ for $ \psi \in L^{2}(\s^{-1})\cap \E(A)$.
Since for some $ \varepsilon >0,$ for $ u_{1}\in [-\varepsilon ,\varepsilon ],$
\begin{align*}
\int_{B((1 + u_{1},0),\varepsilon )}\frac{1}{\s(u)}du \geqslant \int_{ [1 + u_{1}-\varepsilon ,1 + u_{1} + \varepsilon ]\times [-\varepsilon ,\varepsilon ]}cu_{2}^{-2}du_{1}du_{2} =  \infty ,
\end{align*}
$ \psi $ must vanish on $ \{(1 + u_{1},0); | u_{1} | \leqslant \varepsilon \}$.
  Hence $ \lambda :v\mapsto \psi (v,0)$ is an entire function vanishing on an interval. Therefore $ \lambda \equiv 0$, in particular $ \psi (0) = 0$.
  Lemma \ref{lm:ipp-deriv} then yields the conclusion.
%

       \subsection{   Proof of Theorem \ref{thm:converse-rigidity}}
     \label{sec:prf-cvs-main}

%

 {\bf Isotropic case}

We prove that there is no $ Q$-rigidity for some polynomial $ Q = \sum_{\m\preceq \k_{0}}a_{\m}u^{\m}$ for some $ \k_{0}\succeq \k.$ For $ \k$-rigidity, simply take $ Q(u) = u^{\k}.$
Since $ \M$ is not LMR, there exists $ \psi _{0}\in L^{2}(\s^{-1})\cap \E(A)  \setminus \{0\}$. 
The proof relies on a self-contained lemma, which will be useful elsewhere.

\begin{lemma}
For $ \s$ isotropic, if there is $ \psi _{0}\in L^{2}(\s^{-1})\cap \E(B(0,R))  \setminus \{0\}$, then for $ \eta >0$ there is $ \psi \in L^{2}(\s^{-1})\cap \E(B(0,R + \eta ))  \setminus \{0\}$ that is isotropic.
\end{lemma}

\begin{proof}

Let  $ P$ a polynomial such that $\bar  \psi _{0}(u)\sim P(u)$ as $ u\to 0$ (Lemma \ref{lm:polyn-equiv-0}), and 
\begin{align*}
\psi _{1}(u) = P(u)\psi _{0}(u) J_{d}(u\eta /\deg(P) )^{\deg(P)}.
\end{align*}
By Lemma \ref{lm:psi-to-tilde}, $ \psi _{1}\in L^{2}(\s^{-1})\cap \E(A^{ + \eta }).$

 Then define on $  \mathbb C ^{d}$ the rotational average, using the analytic extension of $ \psi _{1}$ on $  \mathbb C ^{d},$
 \begin{align*}
 \psi (z) = \int_{  O(d)}\psi _{1}({\theta }z)d\theta,z\in  \mathbb C ^{d} 
\end{align*}with the Haar measure on the orthogonal group $ O(d)$ of $ \mathbb{R}^{d}$. Still denote by $ \psi $ its restriction to $ \mathbb{R}^{d}$. We must prove several things:
\begin{itemize}
\item  \underline {   $ \psi $ is     isotropic on $ \mathbb{R}^{d}$:} It is immediate from the definition that for $ \theta \in O(d), z\in  \mathbb C ^{d},\psi (\theta z) = \psi (z).$
\item  \underline {    $ \psi $ is not the null function}: By definition, $ \psi _{1}(u)\sim  | P(u) | ^{2}$ as $ u\to 0$, hence $   \psi $ does not vanish identically in the neighbourhood of $ 0$, so that  $  \psi \not\equiv 0. $
\item  \underline { $ \psi \in L^{2}(\s^{-1})$:}   We have
 with the triangle inequality and the Cauchy-Schwarz inequality in $ L^{2}(\s^{-1})$
 \begin{align*}
\int_{\mathbb{R}^{d}}\frac{    | \psi (u) | ^{2}}{\s_{}(u)}du \leqslant &\int_{\mathbb{R}^{d}}\int_{  O(d)\times  O(d)}\frac{  | \psi_{1} ({\theta }u) |   | \psi _{1}({\theta '}u) | }{\s_{}(u)}d\theta d\theta 'du\\
\leqslant & \int_{  O(d)\times  O(d)}\sqrt{\int_{\mathbb{R}^{d}}\frac{ |  \psi_{1} ({\theta }u)^{2} | }{\s_{}(u)}du}\sqrt{\int_{\mathbb{R}^{d}}\frac{ |  \psi _{1}({\theta '}u)^{2} | }{\s_{}(u)}du}d\theta d\theta '\\
\leqslant & \int_{ O(d)\times O(d)}\left(
\sqrt{\int_{ \mathbb{R}^{d}}\frac{  | \psi _{1}(u) | ^{2}}{\s(u)}du}
\right)^{2}d\theta d\theta '<\infty 
\end{align*}
by isotropy of $ \s$, using $ \psi _{1}\in L^{2}(\s^{-1}).$
\item  \underline { $ \psi $ is analytic:}    
For $z, \zeta \in  \mathbb C ^{d}$, since $ \psi _{1}$ is analytic on $  \mathbb C ^{d}$,
\begin{align*}
\psi (z + \zeta ) = \int_{ O(d)}[\psi _{1}({\theta }z) + \psi _{1}'({\theta }z)\zeta  + O(\zeta )]d\theta  = \psi (z) + \left(
\int_{ }\psi _{1}'({\theta }z)d\theta 
\right)\zeta  + O(\zeta ),
\end{align*}using that the quantities involved in the $ O(\cdot )$ are locally bounded; hence $ \psi $ is analytic as well.
\item \underline { $ \psi \in \E(B(0,R + \eta ))$:} 
Also, for $ z\in \mathbb{C} ^{d},\theta \in O(d),$ $$\| {\theta }z\|^{2}_{  \mathbb C } = \| {\theta }  \mathscr  Rz\|^{2} + \|{\theta }  \mathscr  Iz\|^{2} = \|z\|^{2}_{  \mathbb C }.$$ Using both implications of Theorem \ref{thm:SPW}  on $ A^{ + \eta } =  B(0,R +\eta  )$,
\begin{align*}
 | \psi (z) | \leqslant c\int_{ O(d) }\exp((R + \eta )\|z\|_{  \mathbb C })d\theta 
\end{align*}
hence $ \sp(\psi )\subset B(0,R +\eta  ).$
\end{itemize}Finally, $ \psi $ is  indeed an isotropic element of $ L^{2}(\s^{-1})\cap \E(B(0,R + \eta ))  \setminus \{0\}.$

\end{proof}

 Hence non isotropic terms vanish and for some  $ q\in \mathbb{N},\kappa \neq 0$, $$ \psi (u) = \kappa \|u\|^{2q}(\sum_{l>0}a_{l}\|u\|^{2l}),$$ where $ \|u\|^{2l} = (\sum_{i}  u_{i}  ^{2})^{l}$. Hence $ \tilde \psi (u): = \psi (u) \|u\|^{-2q}  $ is analytic and isotropic as well, with spectrum in $ B(0,R + \eta )$ with Theorem \ref{thm:SPW}, with $ \tilde \psi (0)\neq 0$ (which was the whole point).
%
We then can use directly Lemma  \ref{lm:no-Q-rigid} to conclude that $ \M$ is not $ Q$-rigid on $ B(0,R + \eta )$. 
%

 {\bf Separable case.}   We assume that $ \s$ is not (linearly) maximally rigid, hence there is $ \psi \in L^{2}(\s^{-1})\cap \E(A)  \setminus \{0\}$.
Let us build a separable version. Recall that, as analytic function of $ \E(A)$, $ \psi $ has an extension on $  \mathbb C ^{d}$.  Define for $1\leqslant i\leqslant d, z^{i}\in \mathbb{C}^{d-1},$ 
\begin{align*}
\psi^{z^{i}} _{i}(z_{i}) =\psi (z_{i};z^{i}) ,z_{i}\in  \mathbb C ,
\end{align*}
where $  (z_{i};z^{i})$ consists in $ z_{i}$  at the $ i$-th position, surrounded by the $ d-1$ components of $ z^{i}.$
Let the domain $ D\subset \mathbb{R}^{d-1}$ of $ u^{i}\in \mathbb{R}^{d-1}$ where $ \psi_{i} ^{u^{i}}$ is not identically $ 0$.
We have with Fubini's theorem
\begin{align*}
0<\int_{ }\frac{  | \psi  | ^{2}}{\s_{}} = \int_{ \mathbb{R}^{d-1}}\int_{ \mathbb{R}}\frac{ | \psi^{u^{i}} _{i}(u_{i}) | ^{2}}{\s_{i}(u_{i})} \frac{ 1}{\prod_{j\neq i}\s_{j}(u_{j})}du^{i}du_{i } = \int_{D}\int_{ \mathbb{R}}\frac{  | \psi^{u^{i}} _{i}(u_{i}) | ^{2}}{\s_{i}(u_{i})} \frac{ 1}{\prod_{j\neq i}\s_{j}(u_{j})}du^{i}du_{i } <\infty  
\end{align*}
hence there exists $ u^{i}\in D$ such that $ \psi _{i}^{u^{i}}\in L^{2}(\s_{i}^{-1})  \setminus \{0\}$. Clearly $ \psi _{i}^{u^{i}}$ is analytic as analycity in several complex variables implies analycity in each variable.
Let $ q_{i}\in \mathbb{N},\kappa _{i}\neq 0$ the dominating power and coefficient in $ 0:$ $ \psi _{i}^{u ^{i}}(z)\sim \kappa _{i}z^{q_{i}}$ as $ z\to 0.$ Define $\psi _{i}(z): = z^{-q_{i}}\psi _{i}^{u^{i}}(z),z\in  \mathbb C $, still analytic. Define finally 
\begin{align*}
\tilde \psi (z) = \prod_{i}\psi _{i}(z_{i}), z = (z_{i})\in \mathbb{R}^{d}.
\end{align*}Let $ \varepsilon >0.$
We have $ \tilde \psi (0)\neq 0,$ and for some finite $ \kappa >0$, $ \psi _{i}(z)\leqslant \kappa $ if $  \|z\|_{  \mathbb C } \leqslant \varepsilon$. 
 Since for each $ i,$ $ \psi _{i}^{u^{i}}\in L^{2}(\s_{i}^{-1})$, we have $ \tilde \psi \in L^{2}(\s^{-1},   B(0,\varepsilon )^{c})$ (as functions on $ \mathbb{R}^{d}$).

Let $ A_{R} = [-R,R]^{d}$.  It remains to show $ \tilde \psi \in \E(A_{R }).$ We have 
\begin{align*}
s_{A_{R}}(\zeta ) =R\sup_{x: | x_{i} | \leqslant 1}\sum_{i}x_{i}\zeta _{i} = R\sum_{i} \textrm{sign}(\zeta _{i}) \zeta _{i} = R  \sum_{i} | \zeta _{i} |  ,\zeta \in \mathbb{R}^{d}.
\end{align*} Since $ \psi \in \E(A_{R}),$  for $ z\in  \mathbb C ,$ 
\begin{align*}
 | \psi _{i}^{u^{i}}(z) | =  | \psi (z;u^{i}) |  \leqslant &C \exp(s_{A_{R}}(  \mathscr  I (z;u^{i})) )\\
 \leqslant& C\exp( Rc_{u^{i}}  + R|  \mathscr  Iz | )\\
 \leqslant & c'_{u^{i}} \exp(R |  \mathscr  Iz | ).
\end{align*}
Furthermore, $ \tilde \psi $ is
analytic, satisfies $ \tilde \psi (0)\neq 0$ and for $ \|z\|_{  \mathbb C }>\varepsilon ,$
\begin{align*}
 | \tilde \psi (z) | \leqslant \prod_{i}c'_{i}\exp(R |  \mathscr  Iz_{i} | ) = c'\exp(R\sum_{i} |  \mathscr  Iz_{i} | ) = c'\exp(s_{A_{R}}(  \mathscr  Iz)).
\end{align*}
 Theorem \ref{thm:SPW} again yields that $ \tilde \psi \in \E(A_{R})$. Then one can conclude with Lemma  \ref{lm:no-Q-rigid} that we do not have $ Q$-rigidity  on $ A_{R + \eta }$.

%
%
%

 \subsection{Structure factor of cluster lattices}
 \label{sec:flower-prf}
 
\begin{proof}[Proof of Proposition \ref{lm:flower}]
The strategy is to show that 
\begin{align*}
(2\pi )^{-d}\int_{ }| \hat \gamma  | ^{2} d\S=  \textrm{Var}\left(I_{\Phi _{\mu }}(\gamma )\right)\end{align*}
for $ \gamma $ the indicator function of a bounded measurable set. Since $ \S$ is symmetric, we assume also that test functions $ \gamma $ are symmetric.
 The conclusion then comes from formula  \eqref{eq:phase-formula-SF}. 

Denote by $ \nu_{1}  $ and $ \nu _{2}$ the (finite) first and second order intensity  measures of $ \c_{0}$, so that  in particular, Fubini's theorem yields
\begin{align*}
\mathbf{E}\varphi (u) = \int_{ }e^{-i \langle u,x \rangle}\nu_{1} (dx),  \mathbf{E} | \varphi (u) | ^{2} = \int_{ }e^{-i \langle u,x-y \rangle}\nu _{2}(dx,dy) + \kappa .
\end{align*}
We  perform Fourier transform on tempered measures (see Section \ref{sec:theory}), and recall in particular that  
\begin{align} \label{eq:multipl} \F(\gamma   \star \nu _{1} ) =  \hat \gamma  \F\nu   _{1}
\end{align} is valid because $  \gamma $ has compact support and $ \nu _{1}$ is a non-negative finite measure. $ \F\nu _{1}(u) = \mathbf{E}\varphi (u)$ is actually a bounded function, hence $ \hat \gamma \F\nu _{1}$ and $ \gamma  \star \nu _{1}$ are $ L^{2}$ functions.

We use  the conditional variance formula :
\begin{align*}
  \textrm{Var}&\left(I_{\Phi _{\mu }}(\gamma )\right)
 = \mathbf{E}\left(  \textrm{Var} \left(\sum_{\m\in \Phi }\sum_{x\in \c_{\m}}\gamma (\m + x)  \Big| \Phi \right) \right) +   \textrm{Var}\left(\mathbf{E}\left(\sum_{\m\in \Phi }\sum_{x\in \c_{\m}}\gamma (\m + x) \Big| \Phi\right) \right)\\
  = &\mathbf{E}\left(\sum_{\m\in \Phi }  \textrm{Var}\left(\sum_{x\in \c_{\m}}\gamma (\m + x)\Big| \Phi \right)\right) \textrm{ by independence of the $ \c_{\m}$ }\\
  &\hspace{2cm}+  \textrm{Var}\left(\sum_{\m\in \Phi }\int_{ }\gamma (\m + x)\nu_{1} (dx)\right)\textrm{ by definition of $ \nu _{1}$ }\\
   = &\int_{ }  \textrm{Var}\left(\sum_{x\in \c_{0}}\gamma (y + x)\right)dy\textrm{ because $ \Phi $ has intensity $ \Leb$}\\
   & \hspace{2cm}+ (2\pi )^{-d}\int_{ }\S_{\Phi }(du) | \widehat{  \gamma  \star \nu_{1}} (u) | ^{2}\textrm{  by  definition of $ \S_{\Phi }$ and \eqref{eq:phase-formula-SF}}\\
    = &\int_{ \mathbb{R}^{d}}\left[
\mathbf{E}\left(\sum_{x\in \c_{0}}\gamma (y + x)\right)^{2}
-\left(
\int_{ }\tau _{y}\gamma \nu_{1} 
\right)^{2}
\right] dy+  
     (2\pi )^{-d}  \int_{ }\S_{\Phi }(du) | \hat \gamma  \F \nu_{1} (u) | ^{2}\textrm{ by  \eqref{eq:multipl} }\\
      = 	&\int_{ }\int_{ }\gamma (y + x)^{2}\nu_{1} (dx)dy  + \int_{ }  \gamma (x + y)\gamma (z + y)\nu _{2}(dx,dz)dy-\int_{ } \gamma   \star \nu _{1}(y)^{2}dy\\
      & \hspace{2cm}+ (2\pi )^{-d}\int_{ }\S_{\Phi }(du) | \hat \gamma (u) | ^{2} | \mathbf{E}\varphi (u)| ^{2}
     \\
       = &\int_{ }\gamma ^{2}\int_{ }\nu_{1}+ 
      \int_{ } \gamma  \star \gamma (x-z)\nu _{2}(dx,dz) + (2\pi )^{-d}\int_{ }\S_{\Phi }(du) | \hat \gamma (u) | ^{2} | \mathbf{E}\varphi (u)| ^{2}\\
& \hspace{2cm}-  { (2\pi )^{-d}}\int_{ }  | \widehat{ \gamma  \star \nu _{1}}(u) | ^{2}du \textrm{ with Plancherel formula in $ L^{2}(\mathbb{R}^{d})$} \\
        = &\kappa \int_{ }   \gamma ^{2}  +  \int_{ }\left(
\int_{ } (2\pi )^{-d}\widehat{ \gamma   \star \gamma }(u)e^{i \langle u,x-z \rangle }du
\right)\nu _{2}(dx,dz) \textrm{ by Fourier inversion in $ L^{2}(\mathbb{R}^{d})$ }\\
&\hspace{2cm}+  (2\pi )^{-d}\int_{ }\S_{\Phi }(du) | \hat \gamma (u) | ^{2} |  \mathbf{E}\varphi (u)| ^{2}-(2\pi )^{-d}\int_{ }  |  \hat \gamma(u) | ^{2}  |  \F \nu_{1}(u) | ^{2}du,\,\\
         = & (2\pi )^{-d}\kappa \int_{ } \hat \gamma ^{2}+ (2\pi )^{-d} \int_{ } | \hat \gamma (u) | ^{2}\left(
\mathbf{E} | \varphi (u) | ^{2}-\kappa 
\right)du\\
& \hspace{2cm}   +  (2\pi )^{-d}\int_{ }\S_{\Phi }(du) | \hat \gamma (u) | ^{2}  | \mathbf{E}\varphi (u)| ^{2}- (2\pi )^{-d}\int_{ }  |  \hat \gamma(u) | ^{2}  | \mathbf{E}\varphi (u)| ^{2}du\\
 = &(2\pi )^{-d}\left[
\int_{ } | \hat \gamma (u) | ^{2}\left(
\mathbf{E} | \varphi (u) | ^{2}- | \mathbf{E}  \varphi (u) | ^{2}
\right)du + \int_{ }\S_{\Phi }(du) | \hat \gamma (u) | ^{2} | \mathbf{E}\varphi (u) | ^{2}
\right].
\end{align*}
Hence we have indeed $ \S =( \mathbf{E} | \varphi (u) | ^{2}- | \mathbf{E}\varphi (u) | ^{2})\Leb +  | \mathbf{E}\varphi (u) | ^{2}\S_{\Phi }.$

\end{proof}     
    \section*{Acknowledgements}
   I am grateful to  Thomas Lebl\'e, David Dereudre, G. Mastrilli and Lo\"ic Thomassey for insights and discussions about Coulomb gases and hyperuniformity. I also wish to thank T. Hartnick which pointed out the possible application to quasicrystals.

\bibliographystyle{plainnat} \
\bibliography{linear-rigidity.bbl}

\end{document}